\documentclass[final]{siamltex}
\usepackage{graphicx}
\usepackage{psfrag}
\usepackage[utf8]{luainputenc}
\usepackage{amsmath}
\usepackage{amssymb}
\usepackage{color}
\usepackage{afterpage}
\usepackage{subfig}
\usepackage{esint}
\newtheorem{remark}[theorem]{Remark}

\usepackage{hyperref}
\hypersetup{
    colorlinks,
    bookmarksopen,
    bookmarksnumbered,
    citecolor=cyan,
    filecolor=cyan,
    linkcolor=blue,
    urlcolor=green
}

\usepackage{xcolor}
\colorlet{linkequation}{red}
\newcommand*{\SavedEqref}{}
\let\SavedEqref\eqref
\renewcommand*{\eqref}[1]{%
  \begingroup
    \hypersetup{
      linkcolor=linkequation,
      linkbordercolor=linkequation,
    }%
    \SavedEqref{#1}%
  \endgroup
}

\title{Neumann-Neumann Waveform Relaxation Algorithm in Multiple subdomains
for Hyperbolic Problems in 1D and 2D}

\author{Bankim C. Mandal \thanks{Department of Mathematical Sciences, Michigan Technological University, USA ({\tt bmandal@mtu.edu}).}
}

\begin{document}

\maketitle

\begin{abstract}
We present a Waveform Relaxation (WR) version of the Neumann-Neumann 
algorithm for the wave equation in space-time.
The method is based on a non-overlapping spatial domain decomposition,
and the iteration involves subdomain solves in space-time with corresponding
interface condition, followed by a correction step. Using a Fourier-Laplace
transform argument, for a particular relaxation parameter, we prove
convergence of the algorithm in a finite number of steps for finite
time intervals. The number of steps depends on the size of the subdomains
and the time window length on which the algorithm is employed. We
illustrate the performance of the algorithm with numerical results,
followed by a comparison with classical and optimized Schwarz WR
methods.
\end{abstract}


\section{Introduction}
We formulate a new variant of Waveform Relaxation (WR) methods based on the Neumann-Neumann
algorithm to solve hyperbolic problems in parallel computer, and present convergence results for the method. The Neumann-Neumann algorithm was introduced for solving elliptic problems by Bourgat et
al. \cite{BouGT}, see also \cite{RoTal} and \cite{TalRoV}. The
iteration involves solving the subdomain problems using Dirichlet interface conditions
in the first step, followed by a correction step involving Neumann interface conditions.
The convergence behavior of the algorithm is now well
understood for elliptic problems, see for example the book \cite{TosWid}.

To solve time-dependent problems in parallel, the following three possible 
classes of domain decomposition techniques exist: 
\begin{itemize}
\item this approach consists of discretizing
the problem uniformly in time with an implicit scheme to obtain a
sequence of elliptic problems, which are then solved by DD methods. For this kind of technique, we refer to \cite{Cai1,Cai2}.
One disadvantage of this approach is that, uniform time step across the whole
domain need to be enforced, which is very restrictive for problems with variable
coefficients or multiple time scales. Also this method is expensive for parallel computation, since one needs to exchange information at each time step of the discretization.

\item in this approach the equation is
first discretized in space, which is called the method of lines, and
then one applies a waveform relaxation algorithm to solve the large system
of ordinary differential equations (ODEs) obtained from the space-discretization
process. Multigrid dynamic iteration \cite{LubOs,JanVand} and multi-splitting
algorithms \cite{JelPo} are some particular examples of this approach.

\item in contrast to the two
classical techniques above, there exist space-time domain decomposition methods,
formulated at the continuous level. Here, instead of discretizing in time
or in space, one decomposes the original spatial domain into smaller
subdomains and considers each subproblem as posed in both space
and time; then the subproblems are solved iteratively communicating
information at the interfaces between subdomains. This permits the use
of different numerical methods in different subdomains. At each iteration, one solves
the space-time subproblem over the entire time interval of interest, before communicating interface data across subdomains. Thus one saves
communication time while computing in parallel computer. For this approach, see \cite{GanStu,GilKel,Gand1,GH1,Martin} for parabolic problems, and \cite{GHN,GH2,GHN2} for hyperbolic problems.  
\end{itemize} 
In this article, we focus on the WR-type algorithm as it allows different discretizations in different space-time subdomains. WR methods have their origin in the work of Picard \cite{Picard} and Lindel\"of
\cite{Lind} in the late 19th century. Lelarasmee, Ruehli and Sangiovanni-Vincentelli \cite{LelRue} rediscovered WR as a parallel method for the solution of ODEs.

In a different viewpoint, the WR-type methods can be seen as an extension of DD methods for elliptic PDEs. The
systematic extension of the classical Schwarz method to time-dependent
parabolic problems was started in
\cite{GanStu,GilKel}; later optimized SWR methods have been introduced to achieve faster convergence or convergence with no overlap, see \cite{GH1} for parabolic problems, and
\cite{GHN} for hyperbolic problems. Recently, we extended the
substructuring methods, namely the Dirichlet-Neumann and Neumann-Neumann methods, to space-time problems; for
parabolic problems see \cite{GKM1,Mandal,Kwok,GKM2,BankThes}, and for hyperbolic problems see \cite{GKM3,GKM2,BankThes}. We analyzed for the heat
equation to prove that on finite time intervals, the Dirichlet-Neumann Waveform
Relaxation (DNWR) and the Neumann-Neumann Waveform Relaxation (NNWR)
methods converge superlinearly for an optimal choice of the relaxation
parameter. On the contrary for the wave equation, these methods with a two-subdomains decomposition converge in a finite no of steps, see \cite{GKM2}. In this paper, we propose the NNWR method with many subdomains decomposition for hyperbolic problems and analyze the method
for the Wave equation in one and two space dimensions. We analyze
the method in the continuous setting to ensure the understanding
of the asymptotic behavior of the methods in the case of fine grids.

We consider the following hyperbolic PDE on a bounded domain $\Omega\subset\mathbb{R}^{d},0<t<T$,
$d=1,2,3$, with a smooth boundary as our guiding example,
\begin{equation}
\begin{array}{rcll}
\frac{\partial^{2}u}{\partial t^{2}}-c^{2}(\boldsymbol{x})\Delta u & = & f(\boldsymbol{x},t), & \boldsymbol{x}\in\Omega,0<t<T,\\
u(\boldsymbol{x},0) & = & u_{0}(\boldsymbol{x}), & \boldsymbol{x}\in\Omega,\\
u_{t}(\boldsymbol{x},0) & = & v_{0}(\boldsymbol{x}), & \boldsymbol{x}\in\Omega,\\
u(\boldsymbol{x},t) & = & g(\boldsymbol{x},t), & \boldsymbol{x}\in\partial\Omega,0<t<T,
\end{array}\label{modelwave}
\end{equation}
with $c(\boldsymbol{x})$ being a positive function.

We introduce in Section \ref{Section2} the non-overlapping NNWR algorithm with
multiple subdomains for the model problem \eqref{modelwave}, and then 
analyze its convergence for the one dimensional wave equation. 
In Section \ref{Section3} we present convergence result of the NNWR
for multiple subdomains in 2D. Our convergence analysis shows
that both the NNWR algorithm converge in a finite no of steps for
finite time intervals, $T<\infty$. Finally we present numerical results in Section \ref{Section4}, which illustrate our
analysis.

\section{NNWR for multiple subdomains}\label{Section2}

In this section we define the Neumann-Neumann Waveform Relaxation (NNWR)
method with many subdomains for the model problem \eqref{modelwave} on the space-time
domain $\Omega\times(0,T)$ with Dirichlet data given on $\partial\Omega$.
This can be treated as a generalization of the NNWR algorithm for
two subdomains, for which see \cite{GKM3}. 
The method starts with a non-overlapping spatial domain
decomposition, and the iteration involves subdomain solves in space
time with corresponding interface condition, followed by a correction
step.

\subsection{NNWR algorithm}

Suppose the spatial domain $\Omega$ is partitioned into $N$ non-overlapping subdomains $\{\Omega_{i}$,
$1\leq i\leq N\}$, as illustrated in the left panel of Figure \ref{NumFig3}.
For $i=1,\ldots,N$ set $\Gamma_{i}:=\partial\Omega_{i}\setminus\partial\Omega$,
$\Lambda_{i}:=\{j\in\{1,\ldots,N\}:\Gamma_{i}\cap\Gamma_{j}\,\mbox{has nonzero measure}\}$
and $\Gamma_{ij}:=\partial\Omega_{i}\cap\partial\Omega_{j}$, so that
the interface of $\Omega_{i}$ can be rewritten as $\Gamma_{i}=\bigcup_{j\in\Lambda_{i}}\Gamma_{ij}$.
We denote by $u_{i}$ the restriction
of the solution $u$ of \eqref{modelwave} to $\Omega_{i}$ and by $\boldsymbol{n}_{ij}$ the unit outward normal for $\Omega_{i}$
on the interface $\Gamma_{ij}$. The NNWR algorithm for the model
problem \eqref{modelwave} starts with an initial guess $w_{ij}^{0}(\boldsymbol{x},t)$
along the interfaces $\Gamma_{ij}\times(0,T)$, $j\in\Lambda_{i}$,
$i=1,\ldots,N$, and then performs the following two-step iteration:
one first solves Dirichlet subproblems on each $\Omega_{i}$ in parallel, 

\begin{equation}
\begin{array}{rcll}
\partial_{tt}u_{i}^{k}-c^{2}(\boldsymbol{x})\Delta u_{i}^{k} & = & f, & \mbox{in \ensuremath{\Omega_{i}}},\\
u_{i}^{k}(\boldsymbol{x},0) & = & u_{0}(\boldsymbol{x}), & \mbox{in \ensuremath{\Omega_{i}}},\\
\partial_{t}u_{i}^{k}(\boldsymbol{x},0) & = & v_{0}(\boldsymbol{x}), & \mbox{in \ensuremath{\Omega_{i}}},\\
u_{i}^{k} & = & g, & \mbox{on \ensuremath{\partial\Omega_{i}\setminus\Gamma_{i}}},\\
u_{i}^{k} & = & w_{ij}^{k-1}, & \mbox{on \ensuremath{\Gamma_{ij},j\in\Lambda_{i}}}.
\end{array}\label{NNwavemultGen1}
\end{equation}
One then solves Neumann subproblems on all subdomains,
\begin{equation}
\begin{array}{rcll}
\partial_{tt}\varphi_{i}^{k}-c^{2}(\boldsymbol{x})\Delta \varphi_{i}^{k} & = & 0, & \mbox{in \ensuremath{\Omega_{i}}},\\
\varphi_{i}^{k}(\boldsymbol{x},0) & = & 0, & \mbox{in \ensuremath{\Omega_{i}}},\\
\partial_{t}\varphi_{i}^{k}(\boldsymbol{x},0) & = & 0, & \mbox{in \ensuremath{\Omega_{i}}},\\
\varphi_{i}^{k} & = & 0, & \mbox{on \ensuremath{\partial\Omega_{i}\setminus\Gamma_{i}}},\\
\partial_{\boldsymbol{n}_{ij}}\varphi_{i}^{k} & = & \partial_{\boldsymbol{n}_{ij}}u_{i}^{k}+\partial_{\boldsymbol{n}_{ji}}u_{j}^{k}, & \mbox{on \ensuremath{\Gamma_{ij},j\in\Lambda_{i}}}.
\end{array}\label{NNwavemultGen2}
\end{equation}
with the updating condition
\begin{equation}
w_{ij}^{k}(\boldsymbol{x},t)=w_{ij}^{k-1}(\boldsymbol{x},t)-\theta\left(\varphi_{i}^{k}\left|_{\Gamma_{ij}\times(0,T)}\right.+\varphi_{j}^{k}\left|_{\Gamma_{ij}\times(0,T)}\right.\right),\label{NNwavemultGen3}
\end{equation}
where $\theta\in(0,1]$ is a relaxation parameter.
\begin{figure}
  \centering
  \includegraphics[width=0.49\textwidth]{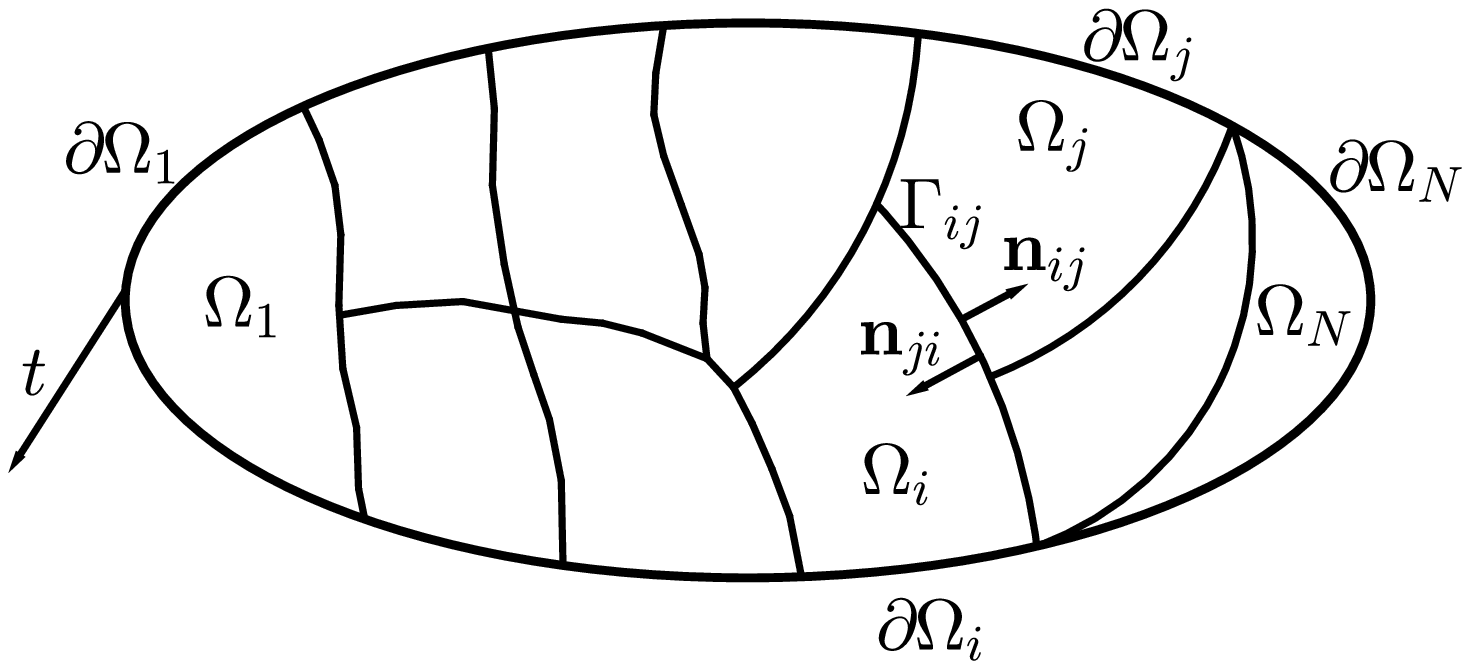}
  \includegraphics[width=0.49\textwidth]{Multidomain.mps}
  \caption{Splitting into many non-overlapping subdomains}
  \label{NumFig3}
\end{figure}

\subsection{Convergence analysis for 1D}

We prove our convergence result for the one dimensional
wave equation with constant speed, $c(x)=c$ on the domain $\Omega:=(0,L)$
with boundary conditions $u(0,t)=g_{0}(t)$ and $u(L,t)=g_{L}(t)$, which in turn become zeros as we consider the error equations, $f(x,t)=0,g_{0}(t)=g_{L}(t)=0=u_{0}(x)=v_{0}(x)$. We decompose $\Omega$ into non-overlapping subdomains $\Omega_{i}:=(x_{i-1},x_{i})$,
$i=1,\ldots,N$, as shown in the right panel of Figure \ref{NumFig3}, and define
the subdomain length $h_{i}:=x_{i}-x_{i-1}$, and $h_{\min}:=\min_{1\leq i\leq N}h_{i}$.
Our initial guess is denoted by $\left\{ w_{i}^{0}(t)\right\} _{i=1}^{N-1}$
on the interfaces $\left\{ x=x_{i}\right\} \times(0,T)$, and for
sake of consistency we denote $w_{0}^{k}(t)=w_{N}^{k}(t)=0$ for all $k$. We then obtain 
\begin{equation}
\begin{array}{rclrcl}
\partial_{tt}u_{i}^{k}-c^{2}\partial_{xx}u_{i}^{k} & = & 0,\qquad\textrm{in \ensuremath{\Omega_{i}}}, & \partial_{tt}\varphi_{i}^{k}-c^{2}\partial_{xx}\varphi_{i}^{k} & = & 0,\qquad\textrm{in \ensuremath{\Omega_{i}}},\\
u_{i}^{k}(x,0) & = & 0,\qquad\textrm{in \ensuremath{\Omega_{i}}}, & \varphi_{i}^{k}(x,0) & = & 0,\qquad\textrm{in \ensuremath{\Omega_{i}}},\\
\partial_{t}u_{i}^{k}(x,0) & = & 0,\qquad\textrm{in \ensuremath{\Omega_{i}}}, & \partial_{t}\varphi_{i}^{k}(x,0) & = & 0,\qquad\textrm{in \ensuremath{\Omega_{i}}},\\
u_{i}^{k}(x_{i-1},t) & = & w_{i-1}^{k-1}(t), & -\partial_{x}\varphi_{i}^{k}(x_{i-1},t) & = & (\partial_{x}u_{i-1}^{k}-\partial_{x}u_{i}^{k})(x_{i-1},t),\\
u_{i}^{k}(x_{i},t) & = & w_{i}^{k-1}(t), & \partial_{x}\varphi_{i}^{k}(x_{i},t) & = & (\partial_{x}u_{i}^{k}-\partial_{x}u_{i+1}^{k})(x_{i},t),
\end{array}\label{NNwavemult1}
\end{equation}
except for the first and last subdomains, where the Neumann conditions
in the Neumann step are replaced by homogeneous Dirichlet conditions
along the physical boundaries. The updated interface values for the
next step will be
\begin{equation}
w_{i}^{k}(t)=w_{i}^{k-1}(t)-\theta\left(\varphi_{i}^{k}(x_{i},t)+\varphi_{i+1}^{k}(x_{i},t)\right).\label{NNwavemult2}
\end{equation}
We start by applying the Laplace transform to the homogeneous Dirichlet
subproblems in \eqref{NNwavemult1}, and obtain 
\[
s^{2}\hat{u}_{i}-c^{2}\hat{u}_{i,xx}=0,\quad\hat{u}_{i}(x_{i-1},s)=\hat{w}_{i-1}(s),\quad\hat{u}_{i}(x_{i},s)=\hat{w}_{i}(s),
\]
for $i=2,...,N-1.$ Set $\gamma_{i}=\cosh\left(h_{i}s/c\right),\sigma_{i}=\sinh\left(h_{i}s/c\right).$
These subdomain problems have the solutions 
\[
\hat{u}_{i}(x,s)=\frac{1}{\sigma_{i}}\left(\hat{w}_{i}(s)\sinh\left((x-x_{i-1})s/c\right)+\hat{w}_{i-1}(s)\sinh\left((x_{i}-x)s/c\right)\right).
\]
Next we apply the Laplace transform to the Neumann subproblems in
\eqref{NNwavemult1} for subdomains not touching the physical boundary,
and obtain 
\[
\hat{\varphi}_{i}(x,s)=C_{i}(s)\cosh\left((x-x_{i-1})s/c\right)+D_{i}(s)\cosh\left((x_{i}-x)s/c\right),
\]
where 
\begin{eqnarray*}
C_{i} & = & \frac{1}{\sigma_{i}}\left(\hat{w}_{i}\left(\frac{\gamma_{i}}{\sigma_{i}}+\frac{\gamma_{i+1}}{\sigma_{i+1}}\right)-\frac{\hat{w}_{i-1}}{\sigma_{i}}-\frac{\hat{w}_{i+1}}{\sigma_{i+1}}\right),\\
D_{i} & = & \frac{1}{\sigma_{i}}\left(\hat{w}_{i-1}\left(\frac{\gamma_{i}}{\sigma_{i}}+\frac{\gamma_{i-1}}{\sigma_{i-1}}\right)-\frac{\hat{w}_{i-2}}{\sigma_{i-1}}-\frac{\hat{w}_{i}}{\sigma_{i}}\right).
\end{eqnarray*}
We therefore obtain for $i=2,...,N-2,$ at iteration $k$ 
\begin{eqnarray*}
\hat{w}_{i}^{k}(s) & = & \hat{w}_{i}^{k-1}(s)-\theta\left(\hat{\varphi}_{i}^{k}(x_{i},s)+\hat{\varphi}_{i+1}^{k}(x_{i},s)\right)\\
 & = & \hat{w}_{i}^{k-1}(s)-\theta\left(C_{i}\gamma_{i}+D_{i}+C_{i+1}+D_{i+1}\gamma_{i+1}\right).
\end{eqnarray*}
Using the identity $\gamma_{i}^{2}-1=\sigma_{i}^{2}$ and simplifying,
we get

\begin{multline}
\hat{w}_{i}^{k}={\displaystyle \hat{w}_{i}^{k-1}-\theta\left(\hat{w}_{i}^{k-1}\left(2+\frac{2\gamma_{i}\gamma_{i+1}}{\sigma_{i}\sigma_{i+1}}\right)+\frac{\hat{w}_{i+1}^{k-1}}{\sigma_{i+1}}\left(\frac{\gamma_{i+2}}{\sigma_{i+2}}-\frac{\gamma_{i}}{\sigma_{i}}\right)\right.}\\
\left.+\frac{\hat{w}_{i-1}^{k-1}}{\sigma_{i}}\left(\frac{\gamma_{i-1}}{\sigma_{i-1}}-\frac{\gamma_{i+1}}{\sigma_{i+1}}\right)-\frac{\hat{w}_{i+2}^{k-1}}{\sigma_{i+1}\sigma_{i+2}}-\frac{\hat{w}_{i-2}^{k-1}}{\sigma_{i}\sigma_{i-1}}\right).\label{eq:5.9}
\end{multline}
For $i=1$ and $i=N$, the Neumann conditions on the physical boundary
are replaced by homogeneous Dirichlet conditions $\varphi_{1}(0,t)=0$
and $\varphi_{N}(L,t)=0$, $t>0$. For these two subdomains, we obtain
as solution after a Laplace transform 
\begin{eqnarray*}
\hat{\varphi}_{1}(x,s) & = & \frac{1}{\gamma_{1}}\left(\hat{w}_{1}\left(\frac{\gamma_{1}}{\sigma_{1}}+\frac{\gamma_{2}}{\sigma_{2}}\right)-\frac{\hat{w}_{2}}{\sigma_{2}}\right)\sinh\left((x-x_{0})s/c\right),\\
\hat{\varphi}_{N}(x,s) & = & \frac{1}{\gamma_{N}}\left(\hat{w}_{N-1}\left(\frac{\gamma_{N-1}}{\sigma_{N-1}}+\frac{\gamma_{N}}{\sigma_{N}}\right)-\frac{\hat{w}_{N-2}}{\sigma_{N-1}}\right)\sinh\left((x_{N}-x)s/c\right),
\end{eqnarray*}
and thus the recurrence relations on the first interface is 
\begin{equation}
\hat{w}_{1}^{k}=\hat{w}_{1}^{k-1}-\theta\left(\hat{w}_{1}^{k-1}\left(2+\frac{\gamma_{1}\gamma_{2}}{\sigma_{1}\sigma_{2}}+\frac{\sigma_{1}\gamma_{2}}{\gamma_{1}\sigma_{2}}\right)+\frac{\hat{w}_{2}^{k-1}}{\sigma_{2}}\left(\frac{\gamma_{3}}{\sigma_{3}}-\frac{\sigma_{1}}{\gamma_{1}}\right)-\frac{\hat{w}_{3}^{k-1}}{\sigma_{2}\sigma_{3}}\right),\label{eq:5.10}
\end{equation}
and on the last interface, we obtain
\begin{multline}
\hat{w}_{N-1}^{k}={\displaystyle \hat{w}_{N-1}^{k-1}-\theta\left(\hat{w}_{N-1}^{k-1}\left(2+\frac{\gamma_{N-1}\gamma_{N}}{\sigma_{N-1}\sigma_{N}}+\frac{\sigma_{N}\gamma_{N-1}}{\gamma_{N}\sigma_{N-1}}\right)\right.}\\
{\displaystyle \left.+\frac{\hat{w}_{N-2}^{k-1}}{\sigma_{N-1}}\left(\frac{\gamma_{N-2}}{\sigma_{N-2}}-\frac{\sigma_{N}}{\gamma_{N}}\right)-\frac{\hat{w}_{N-3}^{k-1}}{\sigma_{N-1}\sigma_{N-2}}\right).}\label{eq:5.11}
\end{multline}
We have the following convergence result for NNWR in 1D:
\begin{theorem}[Convergence of NNWR for multiple subdomains]\label{NNwaveMulti}
Let $\theta=1/4$. Then the NNWR algorithm \eqref{NNwavemult1}-\eqref{NNwavemult2}
converges in $k+1$ iterations, if the time window length $T$ satisfies
$T/k\leq2h_{\min}/c$, $c$ being the wave speed. 
\end{theorem}

\begin{proof}
With $\theta=1/4$ the updating condition \eqref{eq:5.9} becomes
\begin{multline}\label{NNmultiupdate1}
\hat{w}_{i}^{k}(s)=-\frac{1}{4}\left(\hat{t}_{i,i}\hat{w}_{i}^{k-1}(s)+\hat{t}_{i,i+1}\hat{w}_{i+1}^{k-1}(s)+\hat{t}_{i,i-1}\hat{w}_{i-1}^{k-1}(s)\right.\\
\left.- \hat{t}_{i,i+2}\hat{w}_{i+2}^{k-1}(s)-\hat{t}_{i,i-2}\hat{w}_{i-2}^{k-1}(s)\right),
\end{multline}
where we defined $\hat{t}_{i,i}:=\frac{2}{\sigma_{i}\sigma_{i+1}}(\gamma_{i}\gamma_{i+1}-\sigma_{i}\sigma_{i+1})$,
$\hat{t}_{i,i+1}:=\frac{(\sigma_{i}\gamma_{i+2}-\gamma_{i}\sigma_{i+2})}{\sigma_{i}\sigma_{i+1}\sigma_{i+2}}$,
$\hat{t}_{i,i-1}:=\frac{(\sigma_{i+1}\gamma_{i-1}-\gamma_{i+1}\sigma_{i-1})}{\sigma_{i}\sigma_{i-1}\sigma_{i+1}}$,
$\hat{t}_{i,i+2}:=\frac{1}{\sigma_{i+1}\sigma_{i+2}}$, and $\hat{t}_{i,i-2}:=\frac{1}{\sigma_{i}\sigma_{i-1}}$.
Similarly, we obtain for \eqref{eq:5.10} 
\begin{equation}
\hat{w}_{1}^{k}(s)=-\frac{1}{4}\left(\hat{t}_{1,1}\hat{w}_{1}^{k-1}(s)+\hat{t}_{1,2}\hat{w}_{2}^{k-1}(s)-\hat{t}_{1,3}\hat{w}_{3}^{k-1}(s)\right),\label{NNmultiupdate2}
\end{equation}
where we defined $\hat{t}_{1,1}:=\left(\frac{\sigma_{1}\gamma_{2}}{\gamma_{1}\sigma_{2}}+\frac{\gamma_{1}\gamma_{2}}{\sigma_{1}\sigma_{2}}-2\right)$,
$\hat{t}_{1,2}=\frac{1}{\sigma_{2}}\left(\frac{\gamma_{3}}{\sigma_{3}}-\frac{\sigma_{1}}{\gamma_{1}}\right)$
and $\hat{t}_{1,3}=\frac{1}{\sigma_{2}\sigma_{3}}$. From \eqref{eq:5.11},
we obtain 
\begin{equation}
\hat{w}_{N-1}^{k}(s)=-\frac{1}{4}\left(\hat{t}_{N-1,N-1}\hat{w}_{N-1}^{k-1}(s)+\hat{t}_{N-1,N-2}\hat{w}_{N-2}^{k-1}(s)-\hat{t}_{N-1,N-3}\hat{w}_{N-3}^{k-1}(s)\right),\label{NNmultiupdate3}
\end{equation}
where we defined $\hat{t}_{N-1,N-1}=\left(\frac{\sigma_{N-1}\gamma_{N-2}}{\gamma_{N-1}\sigma_{N-2}}+\frac{\gamma_{N-1}\gamma_{N-2}}{\sigma_{N-1}\sigma_{N-2}}-2\right)$, $\hat{t}_{N-1,N-3}=\frac{1}{\sigma_{N-2}\sigma_{N-3}}$
and $\hat{t}_{N-1,N-2}=\frac{1}{\sigma_{N-2}}\left(\frac{\gamma_{N-3}}{\sigma_{N-3}}-\frac{\sigma_{N-1}}{\gamma_{N-1}}\right)$. Note
that $\hat{t}_{i,i+2}=\hat{t}_{i+2,i},\hat{t}_{i,i+1}=-\hat{t}_{i+1,i}$.
So by induction on \eqref{NNmultiupdate1}-\eqref{NNmultiupdate2}
we can write 
\begin{equation}\label{NNupinduct1}
\arraycolsep0.00008em
\begin{array}{rcl}
\hat{w}_{i}^{k}(s)\!\!=\!\!\displaystyle{\sum_{j=-2n}^{2n}}&\left(-\frac{1}{4}\right)^{n}p_{i+j}^{n}\!\left(\hat{t}_{i+j,i+j-2},\hat{t}_{i+j,i+j-1},\!\ldots,\!\hat{t}_{i,i},\!\ldots,\!\hat{t}_{i+j,i+j+1},\hat{t}_{i+j,i+j+2}\right)\hat{w}_{i+j}^{k-n}(s)&,
\end{array}
\end{equation}
and 
\begin{equation}
\hat{w}_{1}^{k}(s)=\sum_{j=0}^{2n}\left(-\frac{1}{4}\right)^{n}p_{1+j}^{n}\left(\hat{t}_{1,1},\ldots,\hat{t}_{1+j,2+j},\hat{t}_{1+j,3+j}\right)\,\hat{w}_{1+j}^{k-n}(s),\label{NNupinduct2}
\end{equation}
where the coefficients $p_{i+j}^{n}$ are homogeneous polynomials
of degree $n$. A similar expression holds for $\hat{w}_{N-1}^{k}(s)$.
Now expanding hyperbolic functions into
infinite binomial series, we obtain
\begin{multline*}
\hat{t}_{i,i}=\frac{2\cosh\left((h_{i}-h_{i+1})s/c\right)}{\sinh(h_{i}s/c)\sinh(h_{i+1}s/c)}=4\left(e^{-2h_{i}s/c}+e^{-2h_{i+1}s/c}\right)\left[1+{\displaystyle \sum_{m=1}^{\infty}}e^{-2h_{i}ms/c}\right.\\
\left.+{\displaystyle \sum_{n=1}^{\infty}}e^{-2h_{i+1}ns/c}+{\displaystyle \sum_{m=1}^{\infty}}{\displaystyle \sum_{n=1}^{\infty}}e^{-2(mh_{i}+nh_{i+1})s/c}\right],
\end{multline*}
 
\begin{multline*}
\hat{t}_{i,i+1}=\frac{\sinh\left((h_{i}-h_{i+2})s/c\right)}{\sinh(h_{i}s/c)\sinh(h_{i+1}s/c)\sinh(h_{i+2}s/c)}=4\left[1+{\displaystyle \sum_{l=1}^{\infty}}e^{-\frac{2lsh_{i}}{c}}+{\displaystyle \sum_{m=1}^{\infty}}e^{-\frac{2msh_{i+1}}{c}}\right.\\
+{\displaystyle \sum_{n=1}^{\infty}}e^{-\frac{2nsh_{i+2}}{c}}+{\displaystyle \sum_{m=1}^{\infty}}{\displaystyle \sum_{n=1}^{\infty}}\left\{ e^{-\frac{2(mh_{i}+nh_{i+1})s}{c}}+e^{-\frac{2(mh_{i+1}+nh_{i+2})s}{c}}+e^{-\frac{2(mh_{i+1}+nh_{i+2})s}{c}}\right\} \\
\left.+{\displaystyle \sum_{l=1}^{\infty}}{\displaystyle \sum_{m=1}^{\infty}}{\displaystyle \sum_{n=1}^{\infty}}e^{-\frac{2(lh_{i}+mh_{i+1}+nh_{i+2})s}{c}}\right]\left(e^{-\frac{(h_{i+1}+2h_{i+2})s}{c}}-e^{-\frac{(h_{i+1}+2h_{i})s}{c}}\right),
\end{multline*}
\begin{multline*}
\hat{t}_{i,i+2}=\frac{1}{\sinh(h_{i+1}s/c)\sinh(h_{i+2}s/c)}=4e^{-(h_{i+1}+h_{i+2})s/c}\left[1+{\displaystyle \sum_{m=1}^{\infty}}e^{-2msh_{i+1}/c}\right.\\
\left.+{\displaystyle \sum_{n=1}^{\infty}}e^{-2nsh_{i+2}/c}+{\displaystyle \sum_{m=1}^{\infty}}{\displaystyle \sum_{n=1}^{\infty}}e^{-2(mh_{i+1}+nh_{i+2})s/c}\right],
\end{multline*}
\begin{multline*}
\hat{t}_{1,1}=\frac{2\cosh\left((2h_{1}-h_{2})s/c\right)}{\sinh(2h_{1}s/c)\sinh(h_{2}s/c)}=4\left(e^{-4h_{1}s/c}+e^{-2h_{2}s/c}\right)\left[1+{\displaystyle \sum_{m=1}^{\infty}}e^{-4msh_{1}/c}\right.\\
\left.+{\displaystyle \sum_{n=1}^{\infty}}e^{-2nsh_{2}/c}+{\displaystyle \sum_{m=1}^{\infty}}{\displaystyle \sum_{n=1}^{\infty}}e^{-2(2mh_{1}+nh_{2})s/c}\right],
\end{multline*}
\begin{multline*}
\hat{t}_{1,2}=\frac{\cosh\left((h_{1}-h_{3})s/c\right)}{\cosh(h_{1}s/c)\sinh(h_{2}s/c)\sinh(h_{3}s/c)}=4\left[1+{\displaystyle \sum_{l=1}^{\infty}}(-1)^{l}e^{-\frac{2lsh_{1}}{c}}+{\displaystyle \sum_{m=1}^{\infty}}e^{-\frac{2msh_{2}}{c}}\right.\\
+{\displaystyle \sum_{n=1}^{\infty}}e^{-\frac{2nsh_{3}}{c}}+{\displaystyle \sum_{m=1}^{\infty}}{\displaystyle \sum_{n=1}^{\infty}}\left\{ (-1)^{m}e^{-\frac{2(mh_{1}+nh_{2})s}{c}}+(-1)^{m}e^{-\frac{2(mh_{1}+nh_{3})s}{c}}+e^{-\frac{2(mh_{2}+nh_{3})s}{c}}\right\} \\
\left.+{\displaystyle \sum_{l=1}^{\infty}}{\displaystyle \sum_{m=1}^{\infty}}{\displaystyle \sum_{n=1}^{\infty}}(-1)^{l}e^{-2(lh_{1}+mh_{2}+nh_{3})s/c}\right]\left(e^{-(2h_{1}+h_{2})s/c}+e^{-(h_{2}+2h_{3})s/c}\right).
\end{multline*}
The argument also holds similarly for the terms $\hat{t}_{i,i-1}$, $\hat{t}_{i,i-2}$, $\hat{t}_{1,3}$, $\hat{t}_{N-1,N-1}$, $\hat{t}_{N-1,N-2}$, $\hat{t}_{N-1,N-3}$.
Now using these expressions we can write \eqref{NNupinduct1}-\eqref{NNupinduct2}
as
\begin{equation}
\hat{w}_{i}^{k}(s)=\left(-1\right)^{k}\left[\left(e^{-2ksh_{i}/c}+e^{-2ksh_{i+1}/c}\right)\hat{w}_{i}^{0}(s)+\sum_{j=-2k}^{2k}q_{i+j}^{k}(s)\,\hat{w}_{i+j}^{0}(s)\right],\label{NNupinduct3}
\end{equation}
and
\begin{equation}
\hat{w}_{1}^{k}(s)=\left(-1\right)^{k}\left[\left(e^{-4h_{1}ks/c}+e^{-2h_{2}ks/c}\right)\hat{w}_{1}^{0}(s)+\sum_{j=0}^{2k}r_{1+j}^{k}(s)\,\hat{w}_{1+j}^{0}(s)\right],\label{NNupinduct4}
\end{equation}
where $q_{i+j}^{k}(s)$ and $r_{1+j}^{k}(s)$ are linear combinations
of terms of the form $e^{-sm}$ with $m\geq2kh_{l}/c$ for some $l\in\left\{ 1,2,\ldots N\right\} $.
A similar expression holds for $\hat{w}_{N-1}^{k}(s)$. We now recall the shifting property of Laplace transform 
\begin{equation}
\mathcal{L}^{-1}\left\{ e^{-\alpha s}\hat{f}(s)\right\} =H(t-\alpha)f(t-\alpha),\label{Lformula}
\end{equation}
where $H(t):=\begin{cases}
1, & t>0,\\
0, & t\leq0.
\end{cases}$ is the Heaviside step function. We use
\eqref{Lformula} to back transform \eqref{NNupinduct3}-\eqref{NNupinduct4}
and obtain
\begin{multline*}
w_{i}^{k}(t)=\left(-1\right)^{k}\left[w_{i}^{0}\left(t-\frac{2kh_{i}}{c}\right)H\left(t-\frac{2kh_{i}}{c}\right)+w_{i}^{0}\left(t-\frac{2kh_{i+1}}{c}\right)H\left(t-\frac{2kh_{i+1}}{c}\right)\right.\\
\left.+\:\textrm{other terms}\right],
\end{multline*}
\begin{multline*}
w_{1}^{k}(t)=\left(-1\right)^{k}\left[w_{1}^{0}\left(t-\frac{4kh_{1}}{c}\right)H\left(t-\frac{4kh_{1}}{c}\right)+w_{1}^{0}\left(t-\frac{2kh_{2}}{c}\right)H\left(t-\frac{2kh_{2}}{c}\right)\right.\\
\left.+\:\textrm{other terms}\right]
\end{multline*}
and a similar expression for $w_{N-1}^{k}(t)$. So for $T\leq2kh_{\min}/c$,
we get $w_{i}^{k}(t)=0$ for all $i$, and the conclusion follows.
\hfill \end{proof}
\begin{remark}\label{Rem1}
The shifting property of Laplace transform \eqref{Lformula} is the reason behind the finite step convergence of the DNWR for a particular value of the parameter $\theta$. 
The right hand side of \eqref{Lformula} becomes identically zero for $t\leq\alpha$, so that for sufficiently small time window length $T$ (e.g., $T\leq\alpha$) the
error becomes zero and leads to convergence in the next iteration.
In Figure \ref{FigDNkernel5} we plot $\mathcal{L}^{-1}\left\{ \hat{f}(s)\right\} $
with $f(t)=\sin(t)$ on the left, and show the effect of time-shifting
on the right.
\begin{figure}
\centering
\includegraphics[width=6cm,height=4cm]{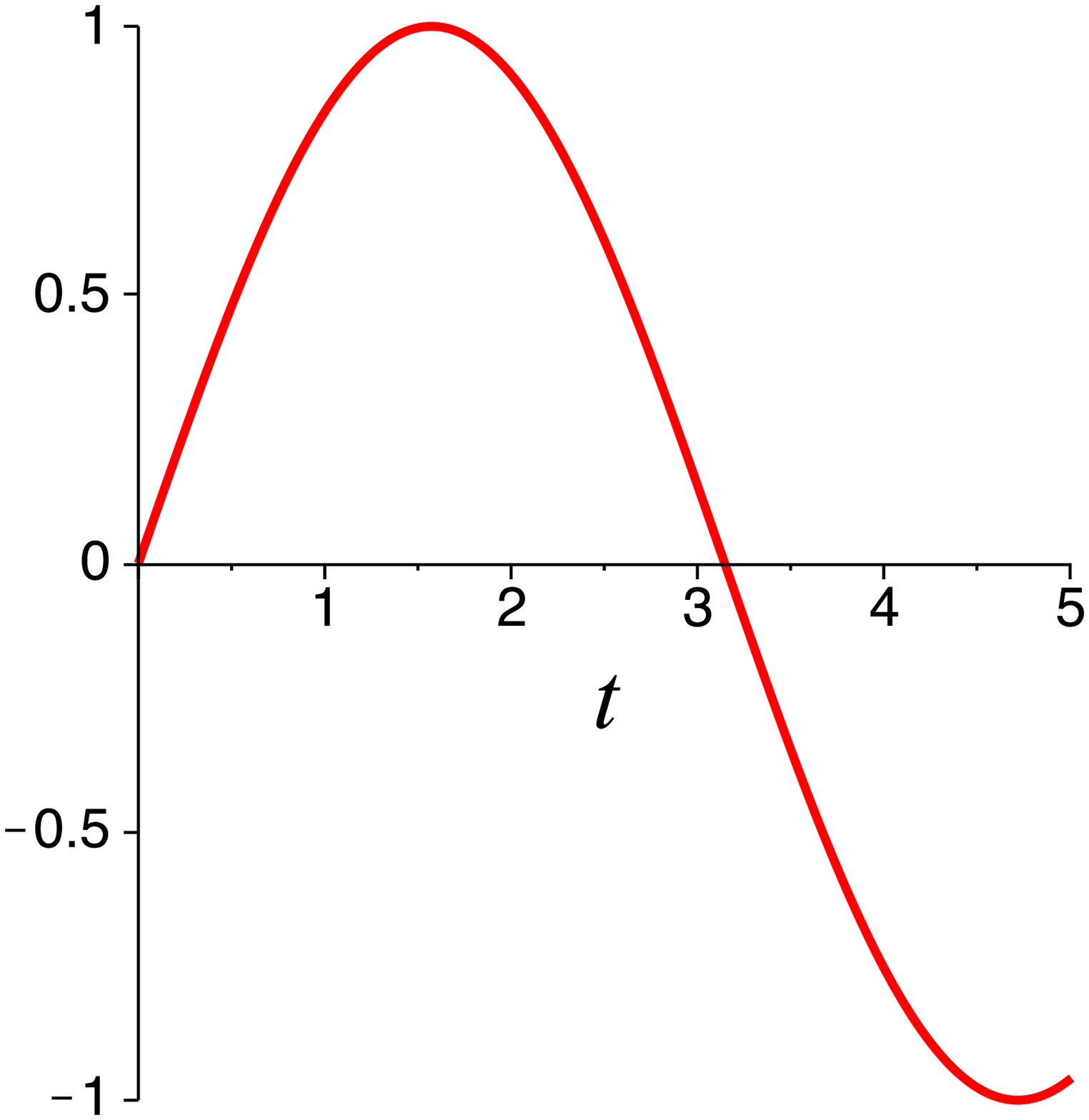}
\includegraphics[width=6cm,height=4cm]{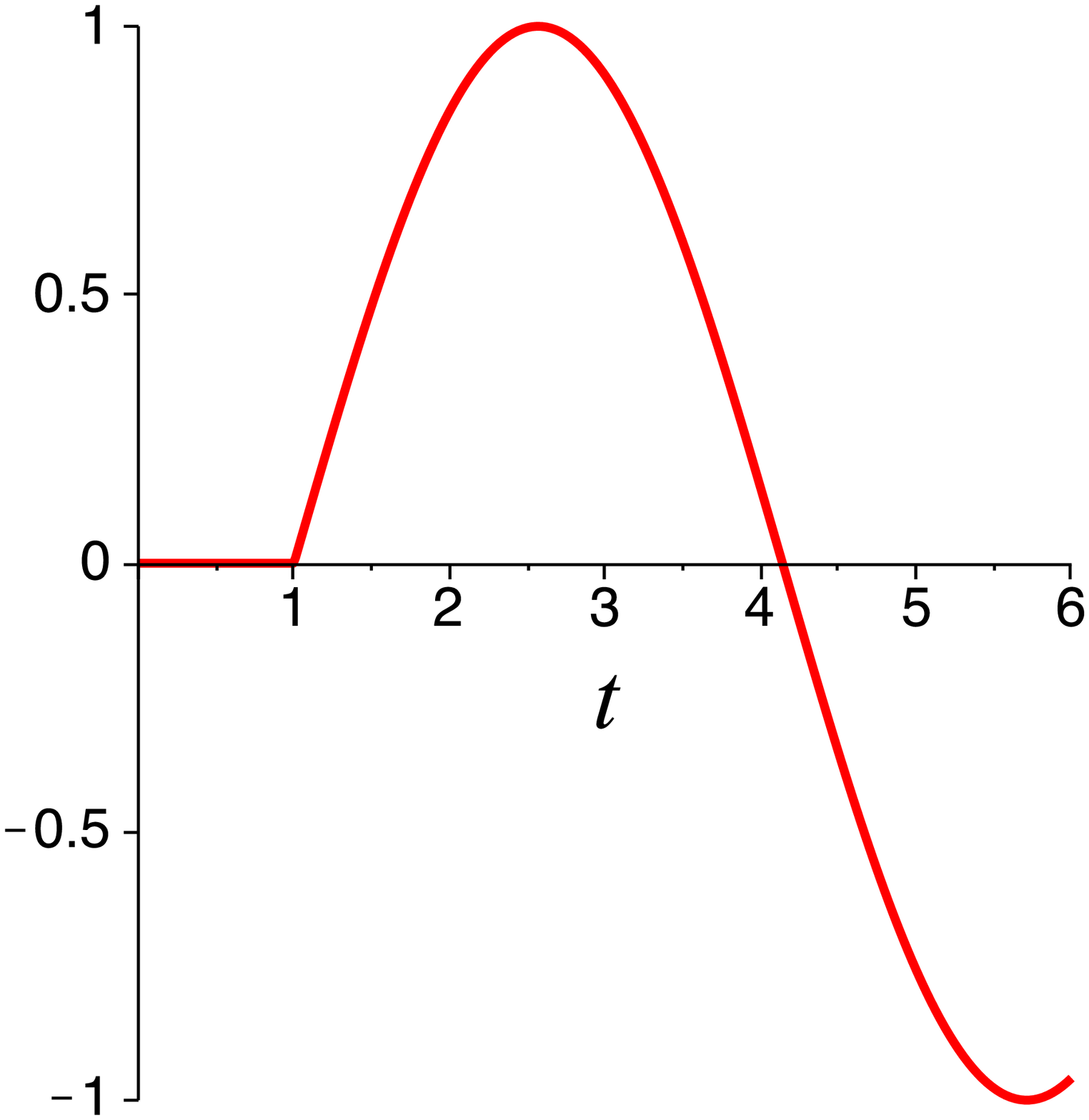}
\caption{Example of time-shifting for the function $f(t)=\sin(t)$: $\mathcal{L}^{-1}\left\{ \hat{f}(s)\right\}$ on the left, and $\mathcal{L}^{-1}\left\{ e^{-s}\hat{f}(s)\right\}$ on the right}
\label{FigDNkernel5}
\end{figure}
\end{remark}

\section{Analysis of NNWR algorithm in 2D}\label{Section3}

We now formulate and analyze the NNWR algorithm for the two-dimensional
wave equation 
\[
\partial_{tt}u-c^{2}\Delta u=f(x,y,t),\quad(x,y)\in\Omega=(l,L)\times(0,\pi),\:t\in(0,T]
\]
with initial condition $u(x,y,0)=u_{0}(x,y),\partial_{t}u(x,y,0)=v_{0}(x,y)$
and Dirichlet boundary conditions. To define the Neumann-Neumann algorithm,
we decompose $\Omega$ into strips of the form $\Omega_{i}=(x_{i-1},x_{i})\times(0,\pi)$,
$l=x_{0}<x_{1}<\cdots<x_{N}=L$. We define the subdomain width $h_{i}:=x_{i}-x_{i-1}$,
and $h_{\min}:=\min_{1\leq i\leq N}h_{i}$. Also we directly consider
the error equations with $f(x,y,t)=0,u_{0}(x,y)=0=v_{0}(x,y)$ and
homogeneous Dirichlet boundary conditions. Given initial guesses $\left\{ w_{i}^{0}(y,t)\right\} _{i=1}^{N-1}$
along the interface $\left\{ x=x_{i}\right\} $, the NNWR algorithm,
as a particular case of \eqref{NNwavemultGen1}-\eqref{NNwavemultGen2}-\eqref{NNwavemultGen3},
is given by performing iteratively for $k=1,2,\ldots$ and for $i=1,\ldots,N$
the Dirichlet and Neumann steps
\begin{equation}\label{NNwave2d}
\arraycolsep0.2em\begin{array}{rclrcl}
\partial_{tt}u_{i}^{k}-c^{2}\Delta u_{i}^{k} & = & 0,\qquad\textrm{in \ensuremath{\Omega_{i}}}, & \partial_{tt}\varphi_{i}^{k}-c^{2}\Delta\varphi_{i}^{k} & = & 0,\qquad\textrm{in \ensuremath{\Omega_{i}}},\\
u_{i}^{k}(x,y,0) & = & 0, & \varphi_{i}^{k}(x,y,0) & = & 0,\\
\partial_{t}u_{i}^{k}(x,y,0) & = & 0, & \partial_{t}\varphi_{i}^{k}(x,y,0) & = & 0,\\
u_{i}^{k}(x_{i-1},y,t) & = & w_{i-1}^{k-1}(y,t), & \partial_{n_{i}}\varphi_{i}^{k}(x_{i-1},y,t) & = & (\partial_{n_{i-1}}u_{i-1}^{k}+\partial_{n_{i}}u_{i}^{k})(x_{i-1},y,t),\\
u_{i}^{k}(x_{i},y,t) & = & w_{i}^{k-1}(y,t), & \partial_{n_{i}}\varphi_{i}^{k}(x_{i},y,t) & = & (\partial_{n_{i-1}}u_{i-1}^{k}+\partial_{n_{i}}u_{i}^{k})(x_{i},y,t),\\
u_{i}^{k}(x,0,t) & = & u_{i}^{k}(x,\pi,t)=0, & \varphi_{i}^{k}(x,0,t) & = & \varphi_{i}^{k}(x,\pi,t)=0,
\end{array}
\end{equation}
except for the first and last subdomain, where in the Neumann step
the Neumann conditions are replaced by homogeneous Dirichlet conditions
along the physical boundaries. The update conditions are defined as
\[
w_{i}^{k}(y,t)=w_{i}^{k-1}(y,t)-\theta\left(\varphi_{i}^{k}(x_{i},y,t)+\varphi_{i+1}^{k}(x_{i},y,t)\right).
\]
We perform a Fourier transform along the $y$ direction to reduce
the original problem into a collection of one-dimensional problems.
Using a Fourier sine series along the $y$-direction, we get 
\[
u_{i}^{k}(x,y,t)=\sum_{n\geq1}U_{i}^{k}(x,n,t)\sin(ny)
\]
 where 
\[
U_{i}^{k}(x,n,t)=\frac{2}{\pi}\int_{0}^{\pi}u_{i}^{k}(x,\eta,t)\sin(n\eta)d\eta.
\]
The NNWR algorithm \eqref{NNwave2d} therefore becomes a sequence
of 1D problems for each $n$,
\begin{equation}
\frac{\partial^{2}U_{i}^{k}}{\partial t^{2}}(x,n,t)-c^{2}\frac{\partial^{2}U_{i}^{k}}{\partial x^{2}}(x,n,t)+c^{2}n^{2}U_{i}^{k}(x,n,t)=0,\label{NNwave2dlap}
\end{equation}
with the boundary conditions for $U_{i}^{k}(x,n,t)$. We now define
\begin{equation}
\chi(\alpha,\beta,t):=\mathcal{L}^{-1}\left\{ \exp\left(-\beta\sqrt{s^{2}+\alpha^{2}}\right)\right\} ,\quad\textrm{Re}(s)>0,\label{Xiab}
\end{equation}
with $s$ being the Laplace variable. Before presenting the main convergence
theorem, we prove the following auxiliary result .
\begin{lemma}\label{Lemma28} We have the identity: 
\[
\chi(\alpha,\beta,t)=\begin{cases}
\delta(t-\beta)-\frac{\alpha\beta}{\sqrt{t^{2}-\beta^{2}}}\,J_{1}\left(\alpha\sqrt{t^{2}-\beta^{2}}\right), & t\geq\beta,\\
0, & 0<t<\beta,
\end{cases}
\]
where $\delta$ is the dirac delta function and $J_{1}$ is the Bessel
function of first order given by 
\[
J_{1}(z)=\frac{1}{\pi}\int_{0}^{\pi}\cos\left(z\sin\varphi-\varphi\right)d\varphi.
\]
 \end{lemma}
\begin{proof}
Using the change of variable $r=\sqrt{s^{2}+\alpha^{2}}$ we write
\[
e^{-\beta r}=e^{-\beta s}-(e^{-\beta s}-e^{-\beta r}).
\]
From the table \cite[p.~245]{Schiff} we get 
\begin{equation}
\mathcal{L}^{-1}\left\{ e^{-\beta s}\right\} =\delta(t-\beta),\label{eq:j1}
\end{equation}
Also on page 263 of \cite{Schiff} we find
\begin{equation}
\mathcal{L}^{-1}\left\{ e^{-\beta s}-e^{-\beta r}\right\} =\begin{cases}
\frac{\alpha\beta}{\sqrt{t^{2}-\beta^{2}}}\,J_{1}\left(\alpha\sqrt{t^{2}-\beta^{2}}\right), & t>\beta,\\
0, & 0<t<\beta.
\end{cases}\label{eq:j2}
\end{equation}
Subtracting \eqref{eq:j2} from \eqref{eq:j1} we obtain the expected
inverse Laplace transform.\hfill\end{proof}

\noindent Now we are ready to prove the convergence result for NNWR in 2D:
\begin{theorem}[Convergence of NNWR in 2D]\label{NNwaveTheorem2D} Let $\theta=1/4$.
For $T>0$ fixed, the NNWR algorithm \eqref{NNwave2d} converges in
$k+1$ iterations, if the time window length $T$ satisfies $T/k<2h_{\min}/c$,
$c$ being the wave speed.
\end{theorem}
\begin{proof}
We take Laplace transforms in $t$ of \eqref{NNwave2dlap} to get
\[
(s^{2}+c^{2}n^{2})\hat{U}_{i}^{k}-c^{2}\frac{d^{2}\hat{U}_{i}^{k}}{dx^{2}}=0,
\]
and now treat each $n$ as in the one-dimensional analysis in the
proof of Theorem \ref{NNwaveMulti}, where the recurrence relations
\eqref{NNmultiupdate1}, \eqref{NNmultiupdate2} and \eqref{NNmultiupdate3}
of the form 
\begin{equation}
\hat{w}_{i}^{k}(s)=\sum_{j}A_{ij}^{(k)}(s)\hat{w}_{j}^{0}(s)\label{Update2dNN}
\end{equation}
now become for each $n=1,2,\ldots$
\begin{equation}
\hat{W}_{i}^{k}(n,s)=\sum_{j}A_{ij}^{(k)}\left(\sqrt{s^{2}+c^{2}n^{2}}\right)\hat{W}_{j}^{0}(n,s).\label{NNwlapup}
\end{equation}
The equation \eqref{Update2dNN} is of the form \eqref{NNupinduct3}-\eqref{NNupinduct4},
that means $A_{ij}^{(k)}(s)$ are linear combination of terms of the
form $e^{-\varrho s}$ for $\varrho\geq2kh_{l}/c$ for some $l\in\left\{ 1,2,\ldots N\right\} $.
Therefore the coefficients $A_{ij}^{(k)}\left(\sqrt{s^{2}+c^{2}n^{2}}\right)$
are sum of exponential functions of the form $e^{-\varrho\sqrt{s^{2}+c^{2}n^{2}}}$
for $\varrho\geq2kh_{l}/c$. Hence we use the definition of $\chi$
in \eqref{Xiab} to take the inverse Laplace transform of \eqref{NNwlapup},
and obtain
\[
W_{i}^{k}(n,t)=\sum_{j}\sum_{m}\chi(cn,\varrho_{m,j},t)*W_{j}^{0}(n,t),
\]
with $\varrho_{m,j}\geq2kh_{\min}/c$. So it is straightforward that
for $t<2kh_{\min}/c$, $W_{i}^{k}(n,t)=0$ for each $n$, since the
function $\chi$ is zero there by Lemma \ref{Lemma28}. Therefore
the interface functions $w_{i}^{k}(y,t)$, given by $w_{i}^{k}(y,t)={\displaystyle {\displaystyle \sum_{n\geq1}}}W_{i}^{k}(n,t)\sin(ny)$
are also zero for all $i\in\left\{ 1,\ldots,N-1\right\} $. Hence
one more iteration produces the desired solution on the entire domain.
\hfill \end{proof}

\section{Numerical Experiments}\label{Section4}

We perform numerical experiments to see the convergence behavior of
the NNWR algorithm with multiple subdomains for the model problem
\begin{align}
\partial_{tt}u & =c^2(x)\partial_{xx}u, &  & x\in(0,5),t>0,\nonumber \\
u(x,0) & =0,\:u_{t}(x,0)=0, &  & 0<x<5,\label{Wavenumericmodel}\\
u(0,t) & =t^{2},\:u(5,t)=t^{2}e^{-t}, &  & t>0,\nonumber 
\end{align}
which is discretized using centered finite differences in both space
and time on a grid with $\Delta x=\Delta t=2{\times}10^{-2}$. We consider a decomposition of $(0,5)$ into five unequal subdomains, whose widths $h_{i}$ are $0.6,0.6,0.5,2.3,1$
respectively, and
take the initial guesses $w_{j}^{0}(t)=t^{2},t\in(0,T]$ for $1\leq j\leq4$. Note that in some
of the experiments below, the coefficient $c(x)$ will be spatially varying. This
will allow us to study how spatially varying coefficients affect the performance of the NNWR, which have only been analyzed in the constant coefficient case. For the first experiment, we take the constant speed, $c=1$. On the left panel of Figure \ref{NumFig7}, we show the convergence
for different values of the parameter $\theta$ for $T=8$, and on
the right the results for the best parameter $\theta=1/4$ for different
time window length $T$. We observe two-step convergence
for $\theta=1/4$ for a sufficiently small time window $T$. Now we take the propagation speed, $c(x)=(x+1)/6$ for the second experiment. On the left panel of Figure \ref{NumFig7b}, we show the convergence
for different values of the parameter $\theta$ for $T=8$, and on
the right the results for the best parameter $\theta=1/4$ for different
time window length $T$.
\begin{figure}
  \centering
  \includegraphics[width=0.49\textwidth]{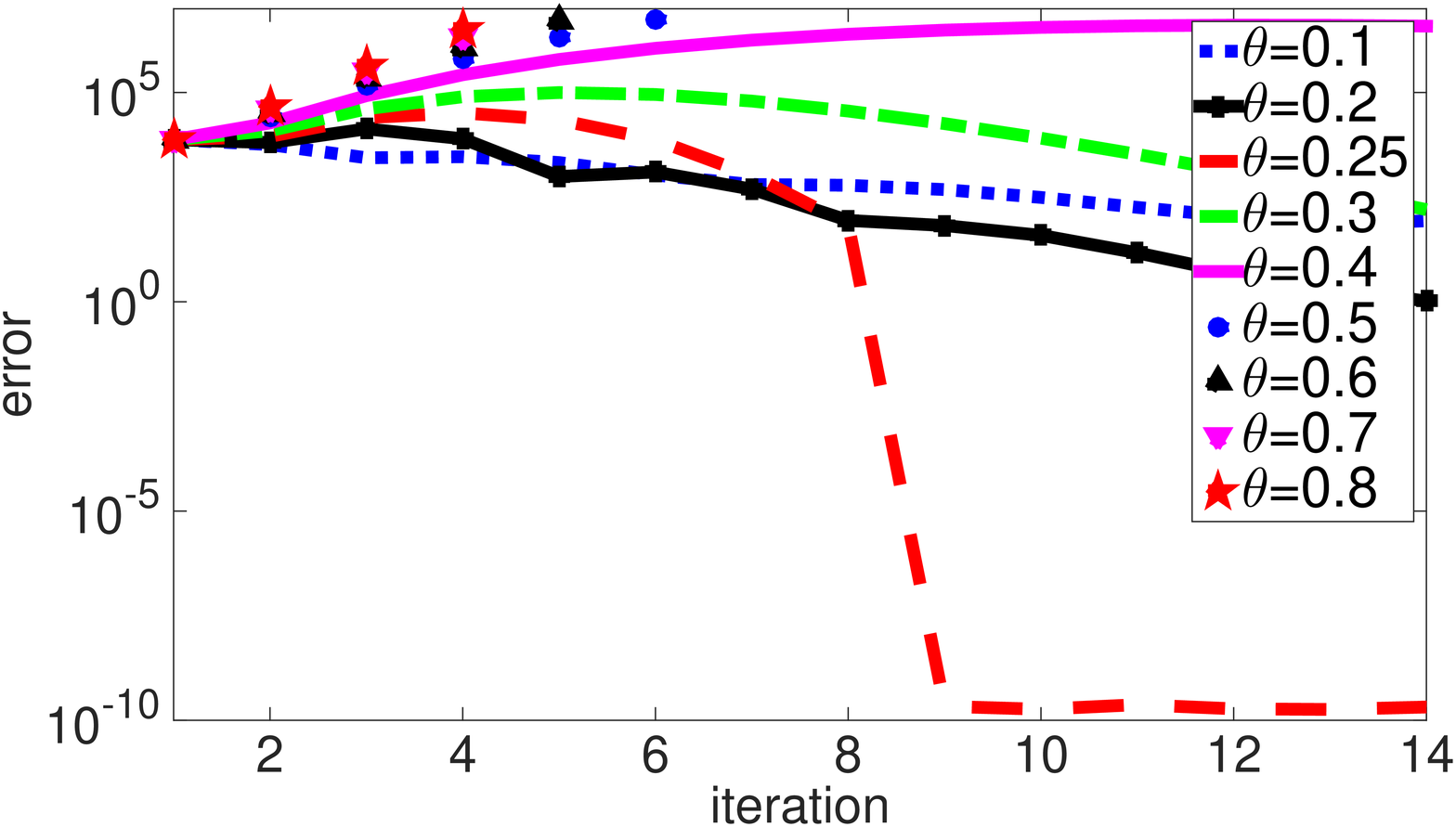}
  \includegraphics[width=0.49\textwidth]{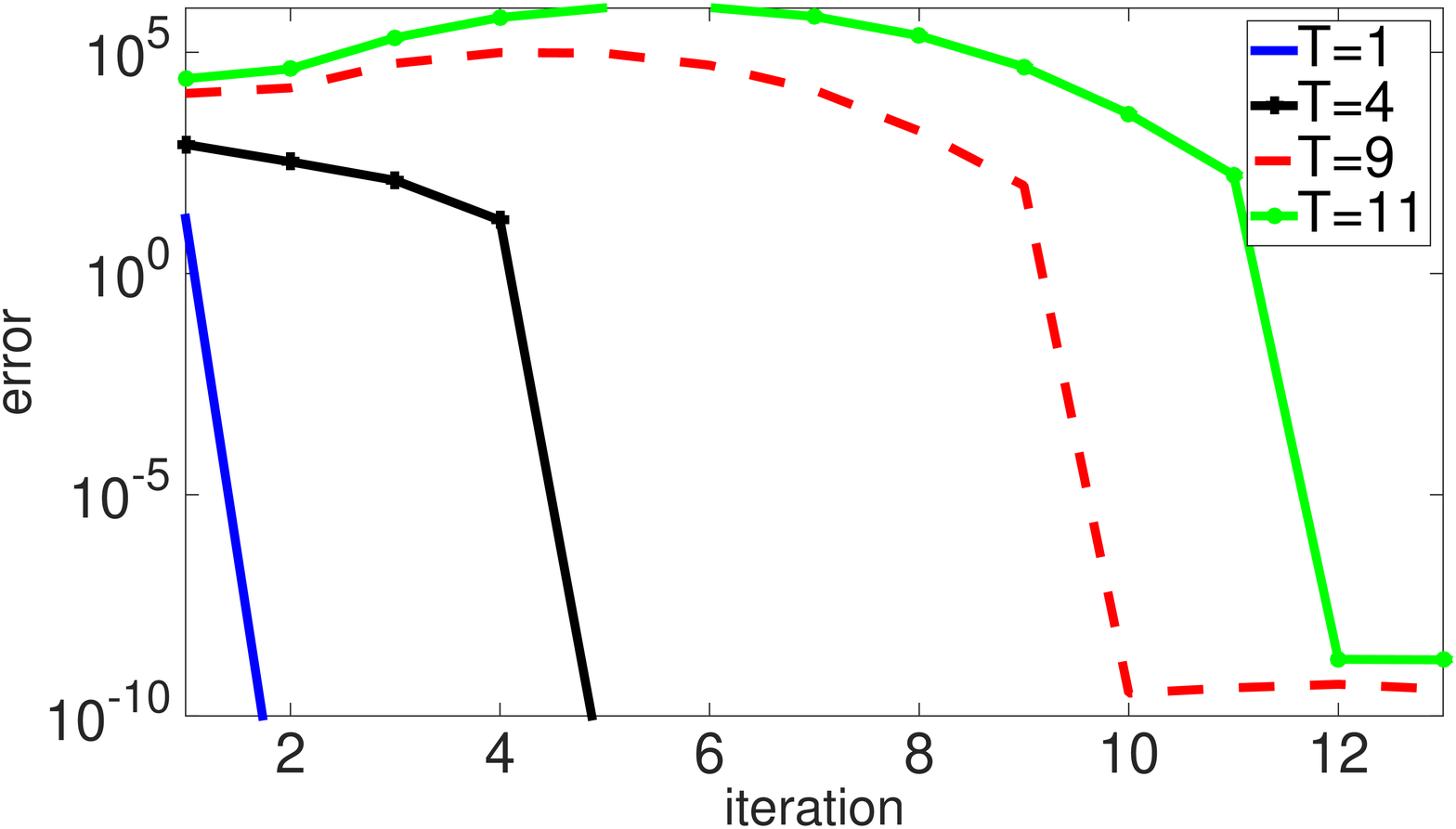}
  \caption{Convergence of NNWR with various values of $\theta$ for $T=8$ on
the left, and for various lengths $T$ of the time window and $\theta=1/4$
on the right}
  \label{NumFig7}
\end{figure} 
\begin{figure}
  \centering
  \includegraphics[width=0.49\textwidth]{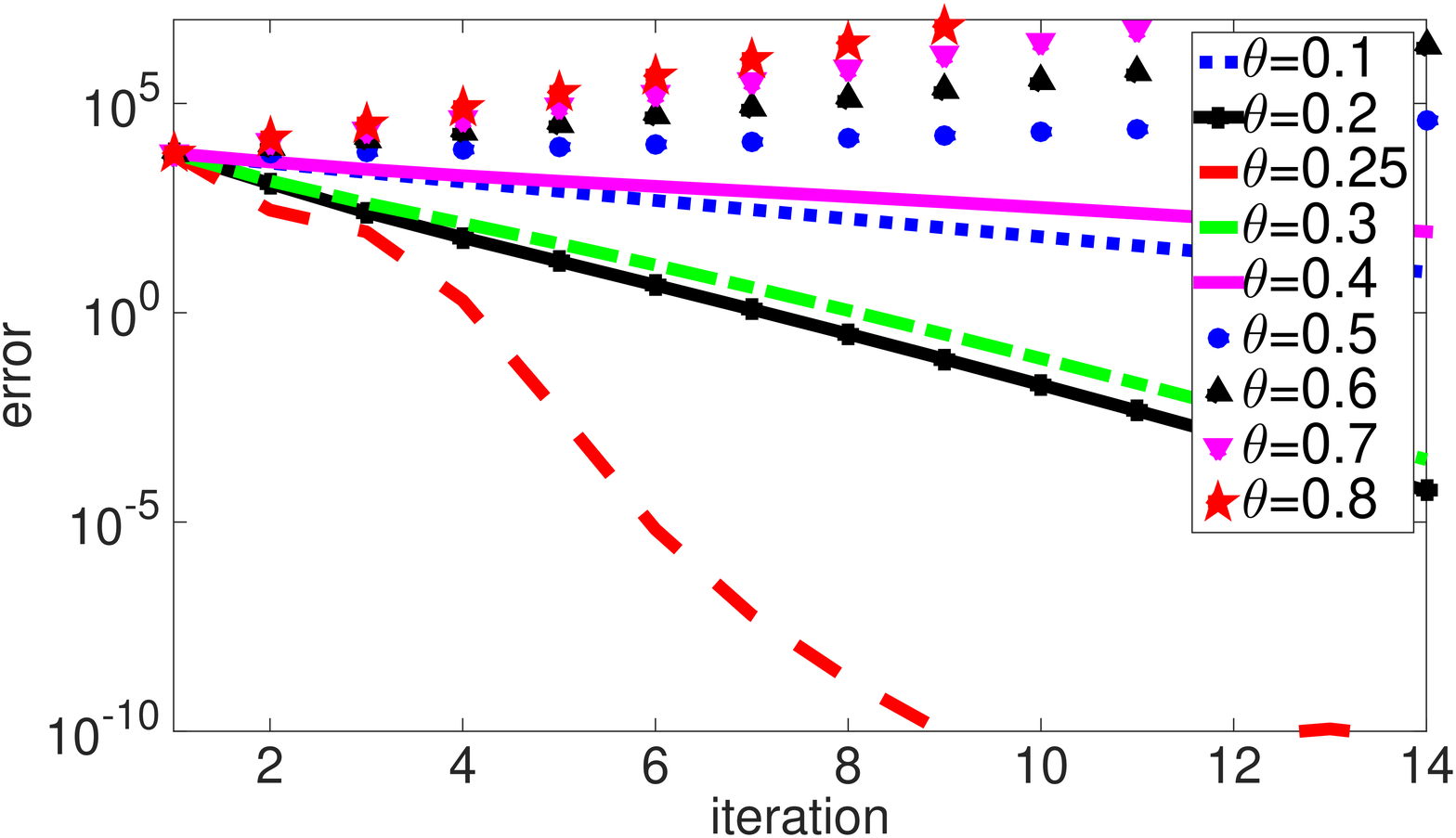}
  \includegraphics[width=0.49\textwidth]{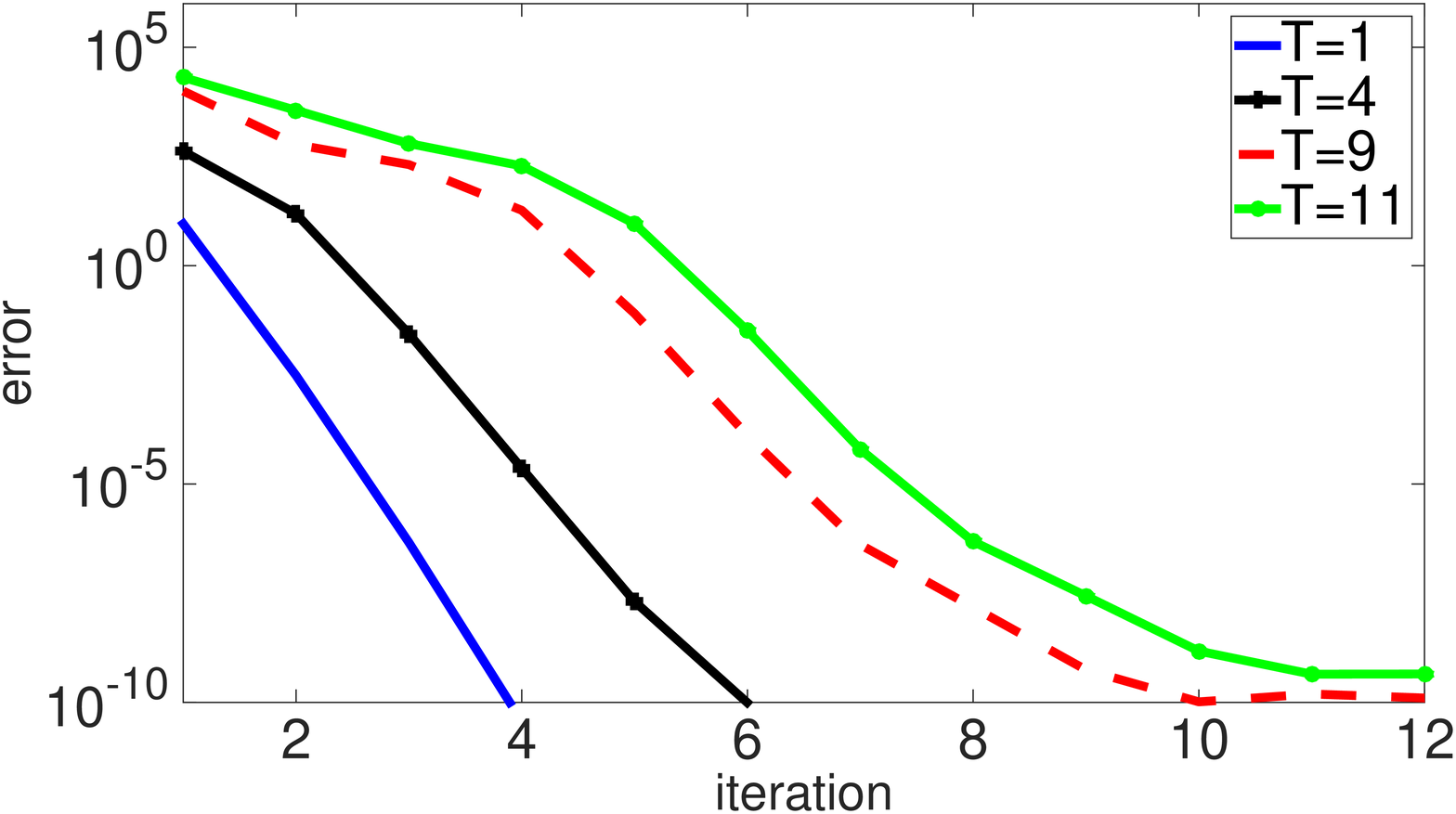}
  \caption{Convergence of NNWR for variable coefficients $c(x)=(x+1)/6$ with various values of $\theta$ for $T=8$ on
the left, and for various lengths $T$ of the time window and $\theta=1/4$
on the right}
  \label{NumFig7b}
\end{figure} 

Next we show an experiment for the NNWR algorithm in two dimension
for the following model problem 
\[
\partial_{tt}u-\left(\partial_{xx}u+\partial_{yy}u\right)=0,u(x,y,0)=xy(x-1)(y-\pi)(5x-2)(4x-3),u_{t}(x,y,0)=0,
\]
with homogeneous Dirichlet boundary conditions. We discretize the
wave equation using the centered finite difference in both space and
time (Leapfrog scheme) on a grid with $\Delta x=5{\times}10^{-2},\Delta y=16{\times}10^{-2},\Delta t=4{\times}10^{-2}$.
We decompose our domain $\Omega:=(0,1)\times(0,\pi)$ into three non-overlapping
subdomains $\Omega_{1}=(0,2/5)\times(0,\pi)$, $\Omega_{2}=(2/5,3/4)\times(0,\pi)$,
$\Omega_{3}=(3/4,1)\times(0,\pi)$. As initial guesses, we take $w_{i}^{0}(y,t)=t\sin(y)$.
In Figure \ref{NumFig8} we plot the convergence curves for different values
of the parameter $\theta$ for $T=2$ on the left panel, and on the
right the results for the best parameter $\theta=1/4$ for different
time window length $T$.
\begin{figure}
  \centering
  \includegraphics[width=0.49\textwidth]{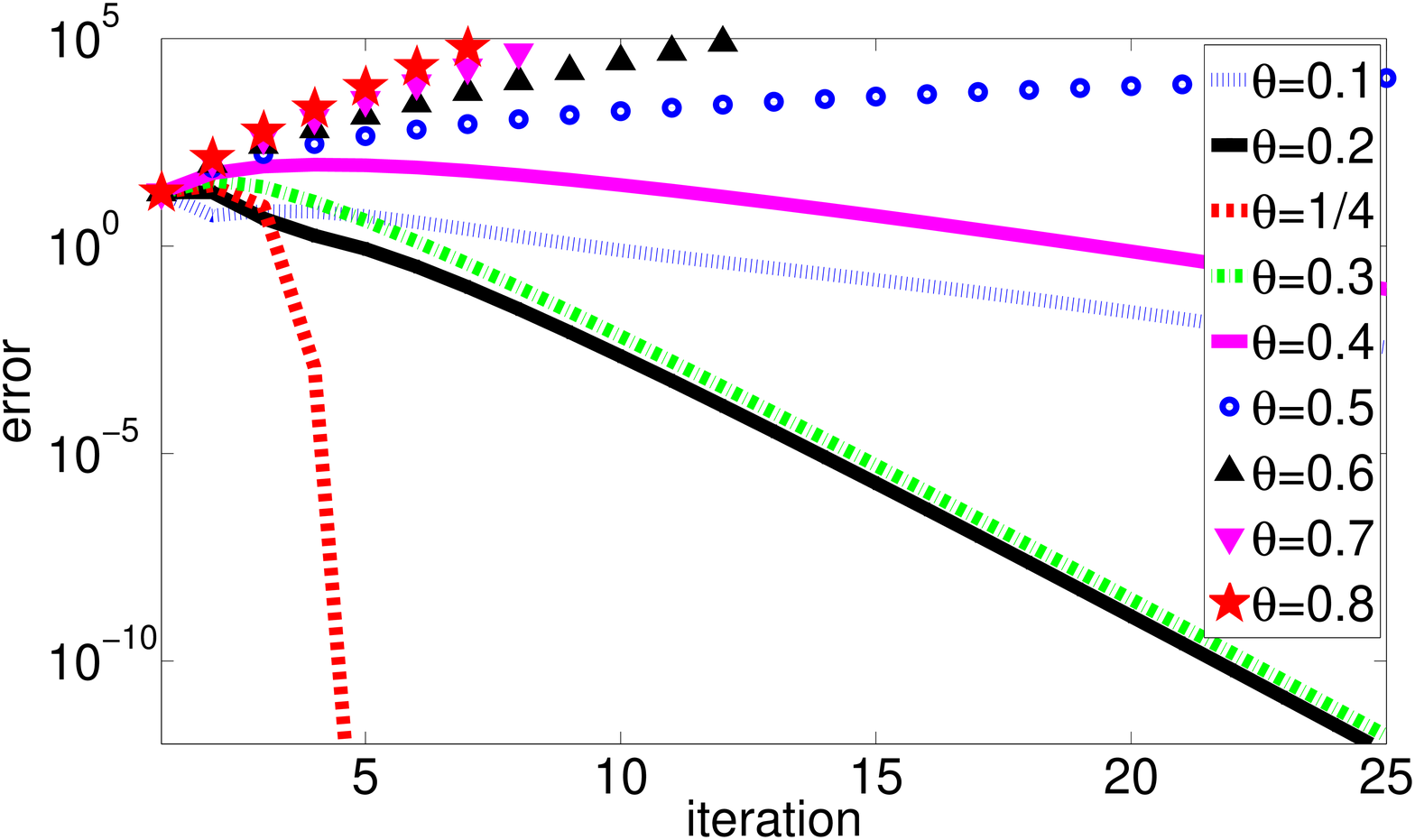}
  \includegraphics[width=0.49\textwidth]{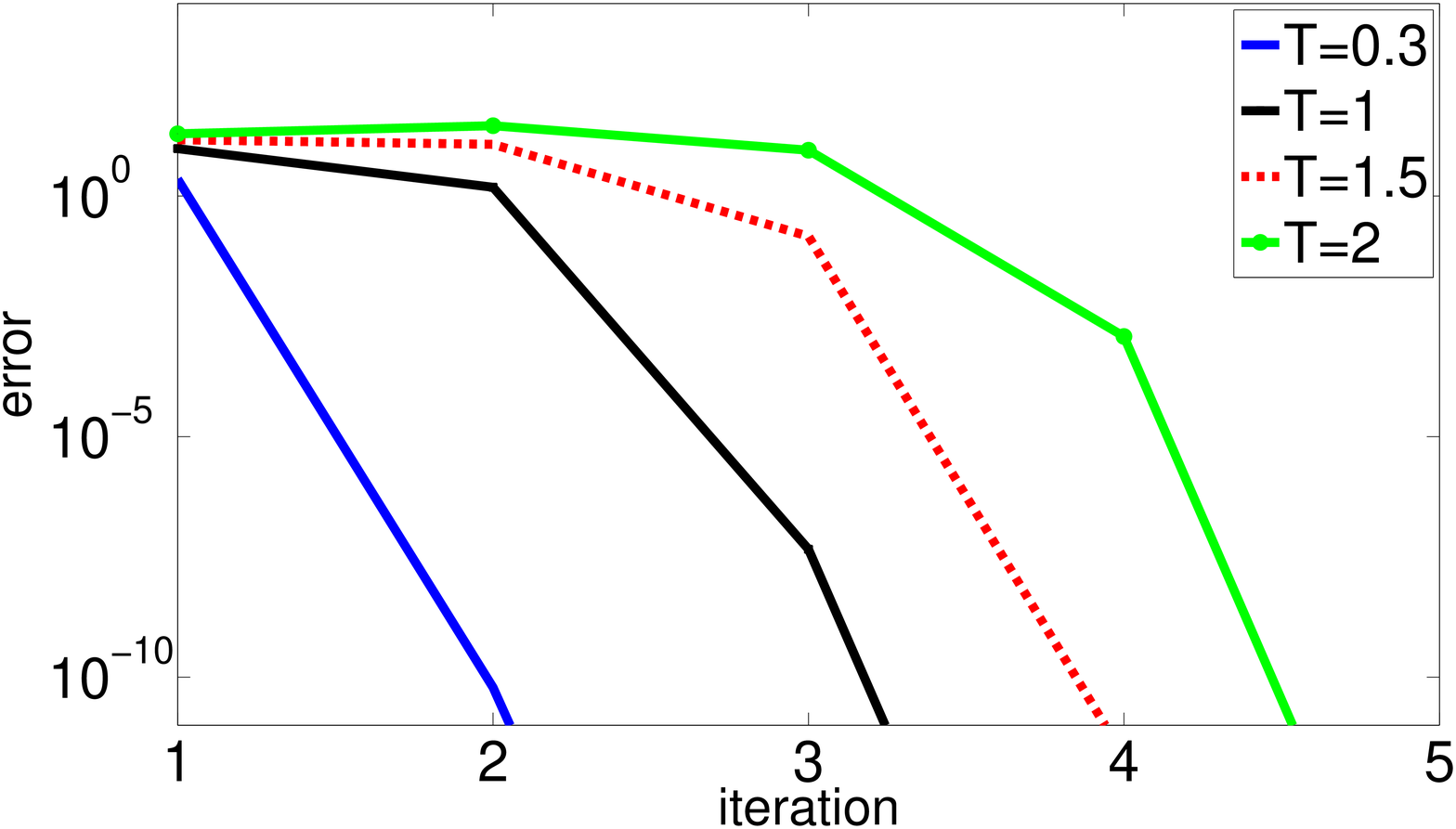}
  \caption{Convergence of NNWR in 2D: curves for different values of $\theta$
for $T=2$ on the left, and for various time lengths $T$ and $\theta=1/4$
on the right}
  \label{NumFig8}
\end{figure}   

We compare in Figure \ref{NumFig9} the performance of the NNWR and DNWR 
(see \cite{GKM3}) algorithms
with the SWR algorithms with and without overlap. Here we consider
the problem 
\begin{align*}
\partial_{tt}u-\partial_{xx}u & =0, &  & x\in(-3,2),t>0,\\
u(x,0) & =0,\:u_{t}(x,0)=xe^{-x}, &  & -3<x<2,\\
u(-3,t) & =-3e^{3}t,\:u(2,t)=2te^{-2}, &  & t>0,
\end{align*}
and for the overlapping Schwarz variant we use an overlap of length
$24\Delta x$, where $\Delta x=1/50$. For the DNWR, NNWR and non-overlapping
SWR we consider a domain decomposition into two subdomains $\Omega_{1}=(-3,0)$
and $\Omega_{2}=(0,2)$. We observe that the DNWR and NNWR algorithms
converge as fast as the Schwarz WR algorithms for smaller time windows
$T$. 
\begin{figure}
  \centering
  \includegraphics[width=0.49\textwidth]{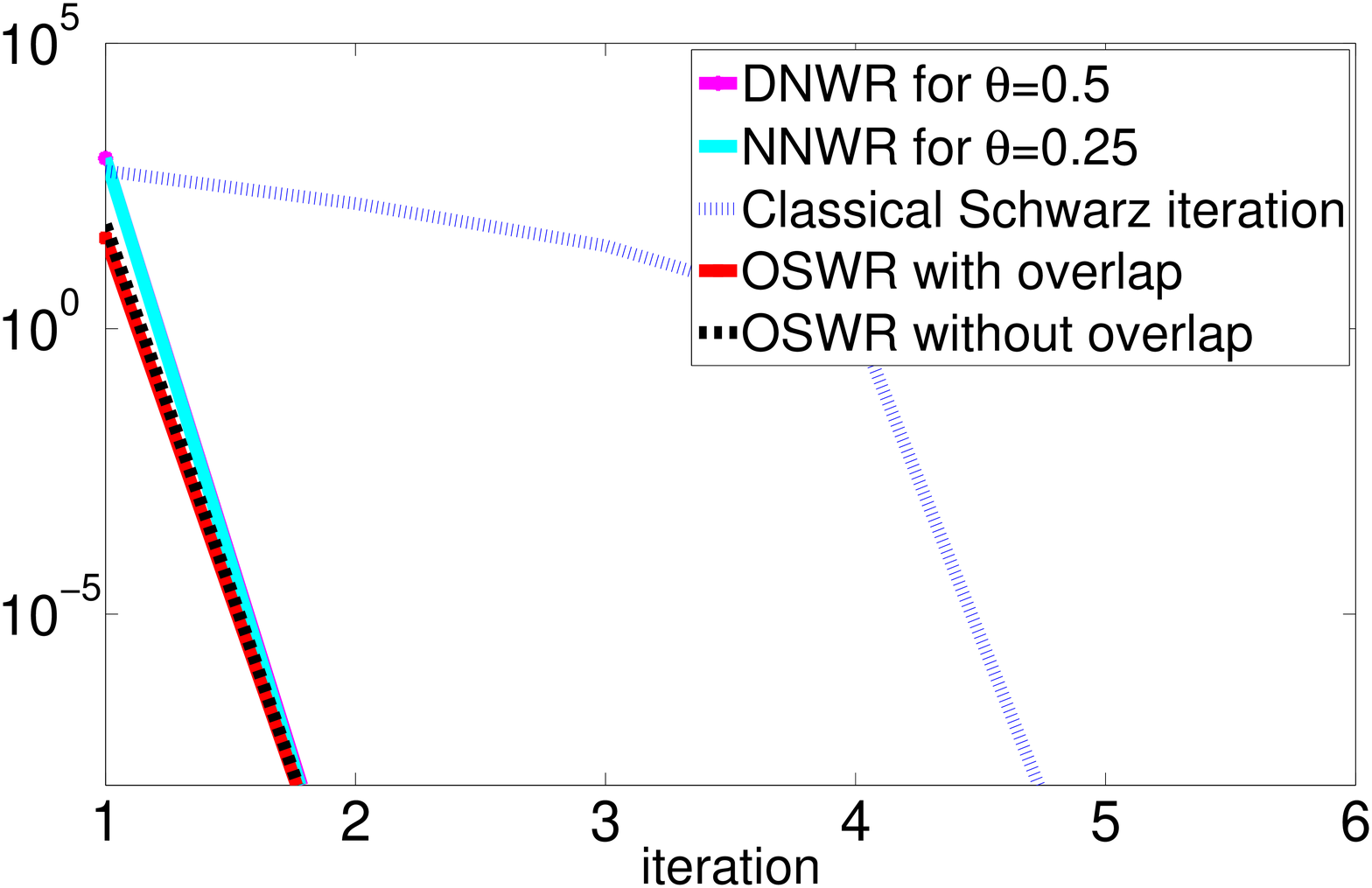}
  \includegraphics[width=0.49\textwidth]{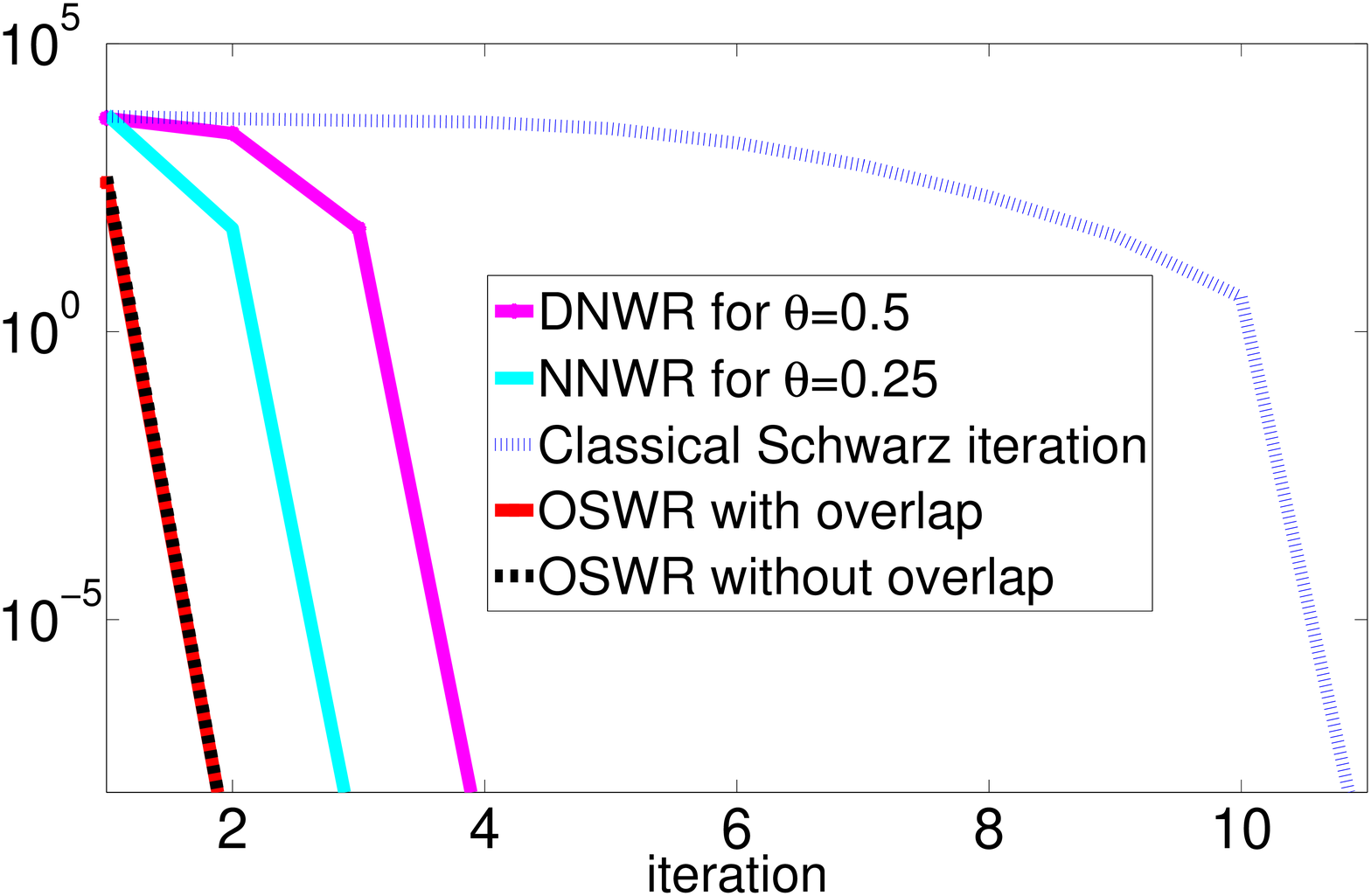}
  \caption{Comparison of DNWR, NNWR, and SWR for 1D wave equation for $T=4$ on the left, and $T=10$ on the right}
  \label{NumFig9}
\end{figure} 
\begin {table}\begin{center} \caption {Comparison of steps needed for convergence for 1D wave equation. \label{TableSWR2}} \begin{tabular}{|c|c|c|c|} \hline Methods & 2 subdomains, 1D & Many subdomains, 1D & Many subdomains, 2D \tabularnewline \hline DNWR & $T\leq 2kh_{\min}/c$ & $T\leq kh_{\min}/c$ & $T< kh_{\min}/c$\tabularnewline \hline NNWR & $T\leq 4kh_{\min}/c$ & $T\leq 2kh_{\min}/c$ & $T< 2kh_{\min}/c$\tabularnewline \hline \end{tabular} \end{center} \end {table} 
Due to the local nature of the Dirichlet-to-Neumann operator in 1D \cite{GHN}, SWR converges in a finite number of iterations just like DNWR and
NNWR. In higher dimensions, however, non-overlapping SWR will no longer
converge in a finite number of steps, but DNWR and NNWR will; see
Figure \ref{NumFig10}. Table \ref{TableSWR2} gives a summary of the theoretical
results from Section 2 and 3 and \cite{GKM3}, to indicate the maximum
number of iterations needed for the 1D and 2D wave equation to converge to
the exact solution. For the comparison result in 2D, we consider the model problem
\[
\partial_{tt}u-\left(\partial_{xx}u+\partial_{yy}u\right)=0,u(x,y,0)=0=u_{t}(x,y,0),
\]
with Dirichlet boundary conditions $u(0,y,t)=t^{2}\sin(y),u(1,y,t)=y(y-\pi)t^{3}$
and $u(x,0,t)=0=u(x,\pi,t)$. We decompose our domain $\Omega:=(0,1)\times(0,\pi)$
for the two subdomains experiment into $\Omega_{1}=(0,3/5)\times(0,\pi)$
and $\Omega_{2}=(3/5,1)\times(0,\pi)$, and for the three subdomains
experiment into $\Omega_{1}=(0,2/5)\times(0,\pi)$, $\Omega_{2}=(2/5,3/4)\times(0,\pi)$,
$\Omega_{3}=(3/4,1)\times(0,\pi)$. We take a random initial guess
to start the iteration, and for the overlapping SWR we use an overlap
of length $2\Delta x$ in all the experiments. We implement first
order methods with one parameter in optimized SWR iterations; for
more details see \cite{GH2}. On the left panel of Figure \ref{NumFig10} we plot
the comparison curves for two subdomains, and the same for three subdomains
on the right. 
\begin{figure}
  \centering
  \includegraphics[width=0.49\textwidth]{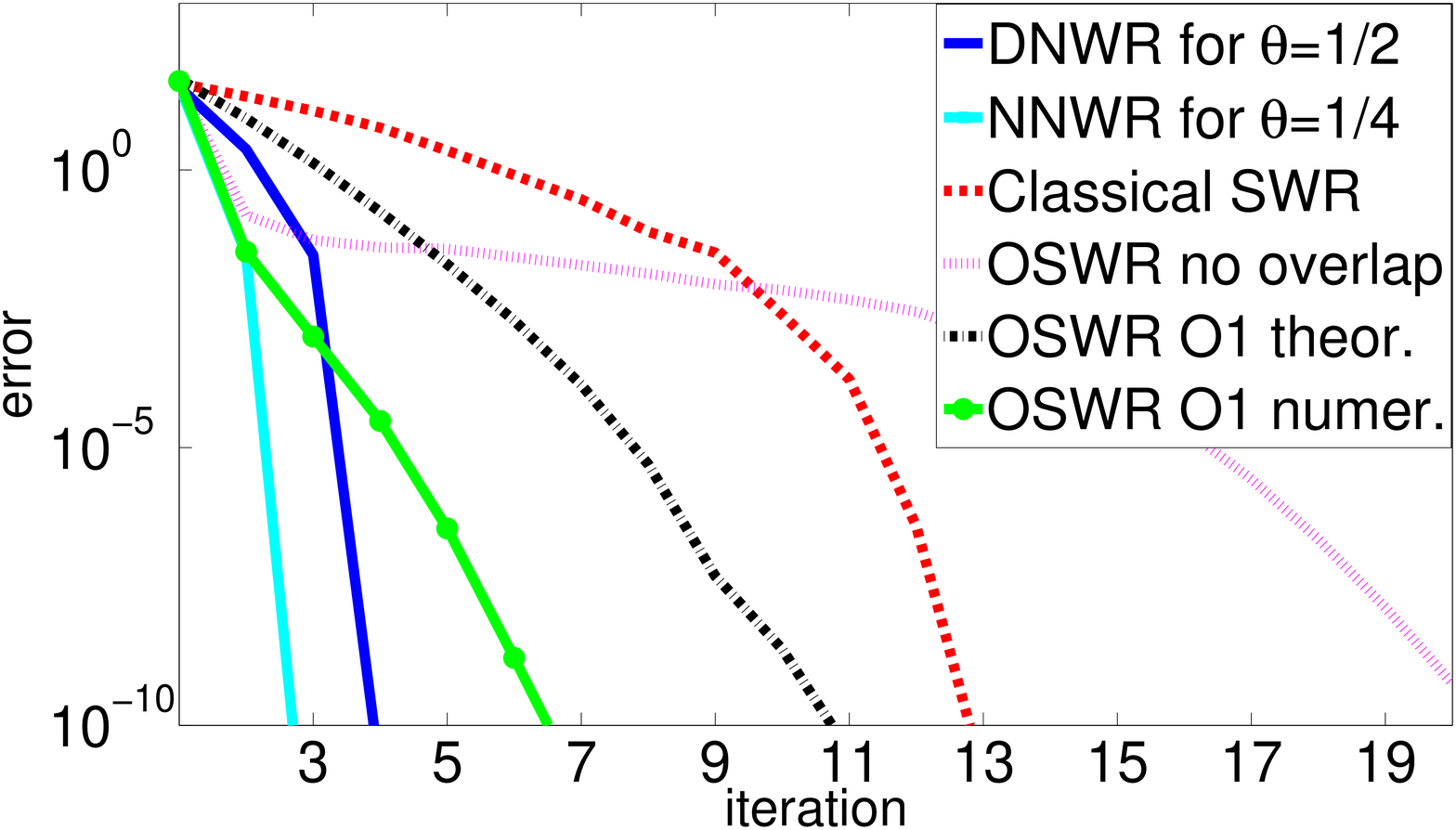}
  \includegraphics[width=0.49\textwidth]{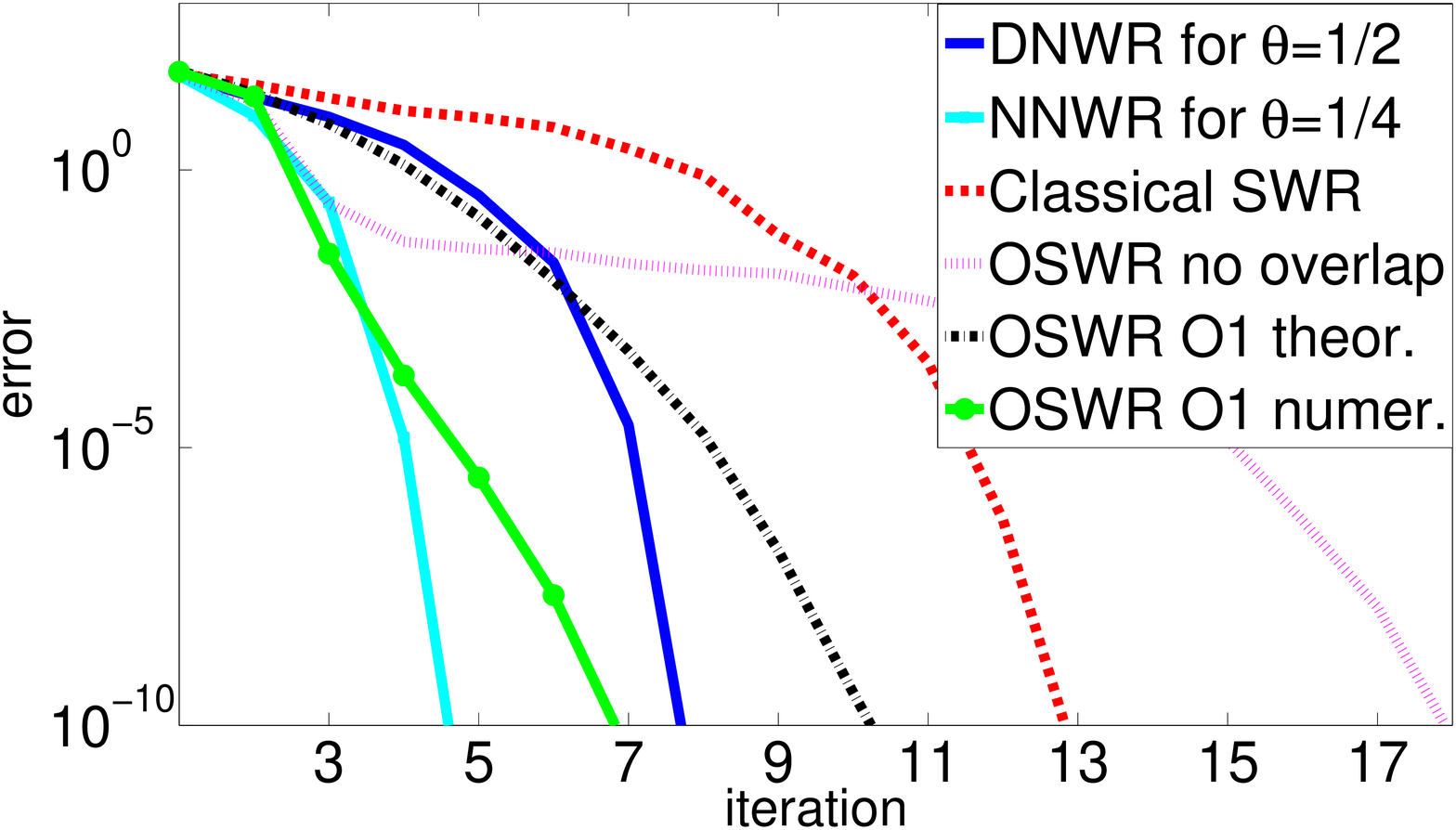}
  \caption{Comparison of DNWR, NNWR, and SWR for $T=2$ in 2D for two subdomains on the left, and three subdomains on the right}
  \label{NumFig10}
\end{figure} 

Now we show a numerical experiment for the NNWR algorithm with different
time grids for different subdomains and discontinuous wave speed across
interfaces. We consider the model problem 
\[
\partial_{tt}u-c^{2}\partial_{xx}u=0,u(x,0)=0=u_{t}(x,0),
\]
with Dirichlet boundary conditions $u(0,t)=t^{2},u(6,t)=t^{3}$. Suppose
the spatial domain $\Omega:=(0,6)$ is decomposed into three equal
subdomains $\Omega_{i},i=1,2,3$, and the random initial guesses are
used to start the NNWR iteration. For the spatial discretization,
we take a uniform mesh with size $\Delta x=1{\times}10^{-1}$, and
for the time discretization, we use non-uniform time grids $\Delta t_{i},i=1,2,3$, as given in Table \ref{Table23}. For the non-uniform mesh grid,
boundary data is transmitted from one subdomain to a neighboring subdomain
by introducing a suitable time projection. For two dimensional problems,
the interface is one dimensional. Using ideas of merge sort one can
compute the projection with linear cost, see \cite{GJap} and the references
therein. Even for higher dimensional interfaces, such an algorithm
with linear complexity is still possible, see \cite{GJap2}. In Figure
\ref{FigDNnonUnif} we show the non-uniform time steps for different
subdomains. Figure \ref{NumFig11} gives the convergence behavior
of the NNWR algorithm for $T=2$ with the same non-uniform time grids
as in Table \ref{Table23}.
\begin {table} \begin{center} \caption {Propagation speed and time steps for different subdomains.}\label{Table23} \begin{tabular}{|c|c|c|c|} \hline & $\Omega_1$ & $\Omega_2$ & $\Omega_3$ \tabularnewline \hline wave speed $c$& $1/4$ & $2$ & $1/2$ \tabularnewline \hline time grids $\Delta t_i$ & $13\times 10^{-2}$ & $39\times 10^{-3}$ & $1\times 10^{-1}$ \tabularnewline \hline \end{tabular} \end{center} \end {table} 
\begin{figure} 
\centering 
\includegraphics[width=0.52\textwidth]{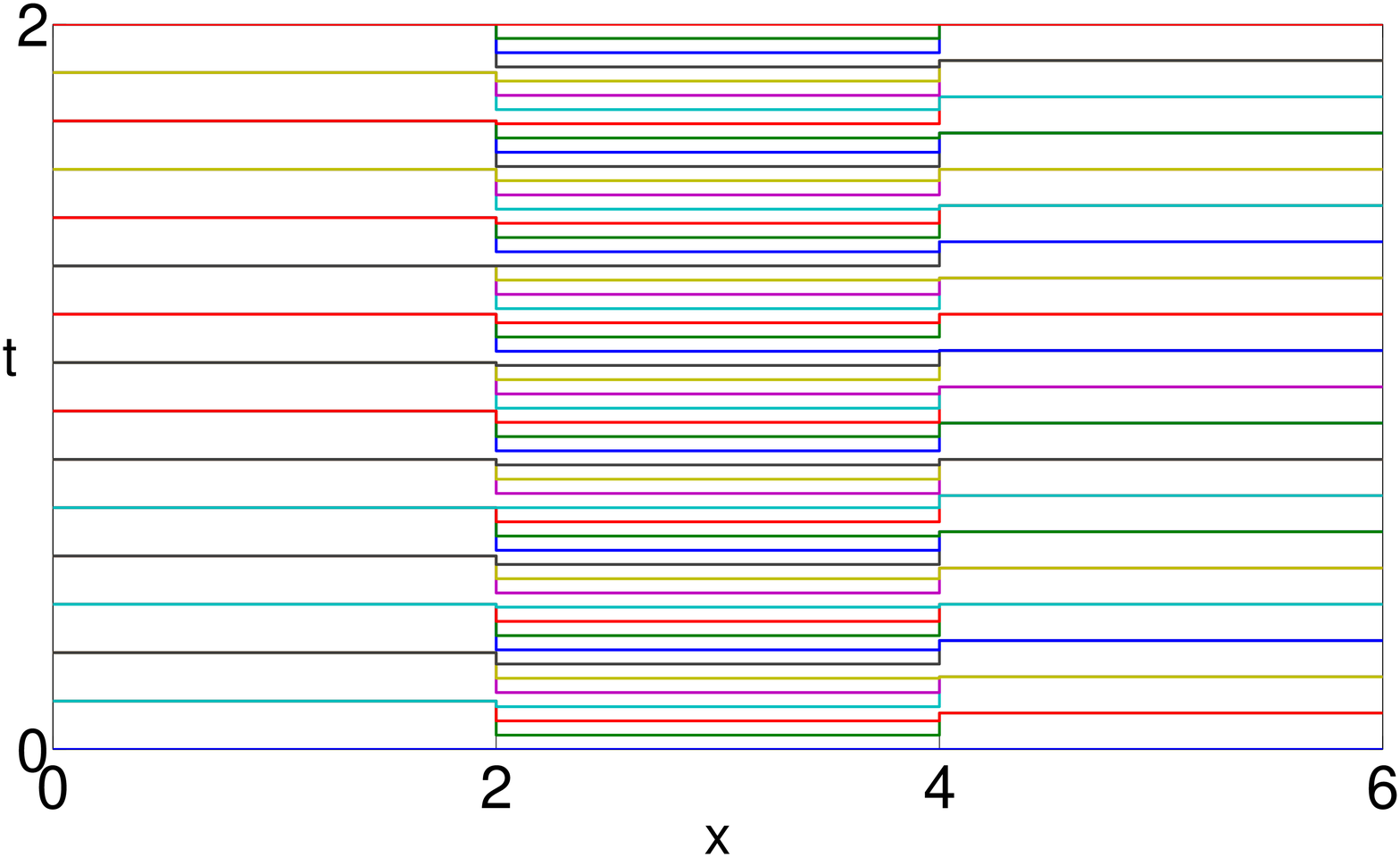} 
\caption{Subdomains with non-uniform time steps} 
\label{FigDNnonUnif} 
\end{figure}
\begin{figure}
  \centering
  \includegraphics[width=0.49\textwidth]{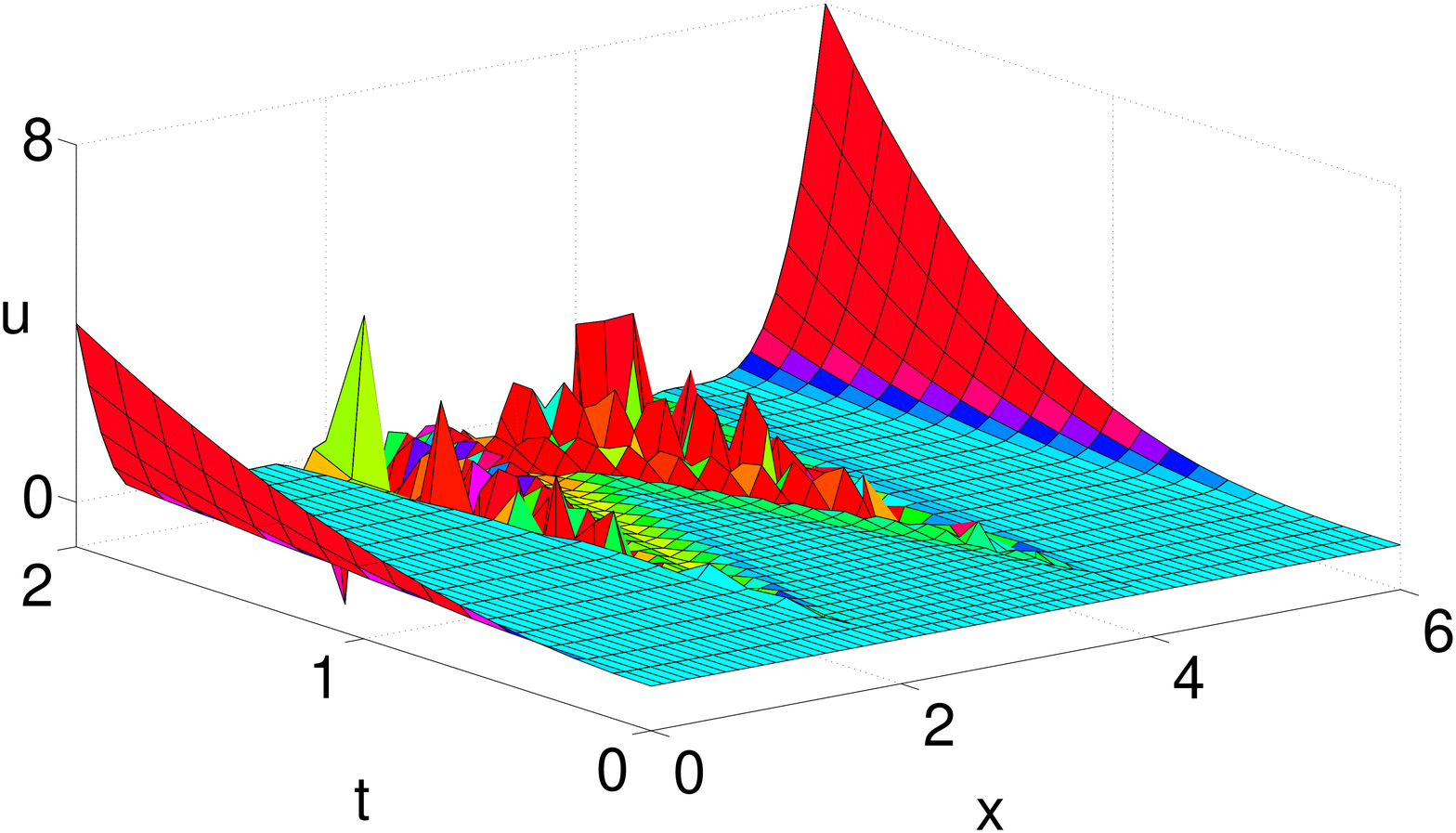}
  \includegraphics[width=0.49\textwidth]{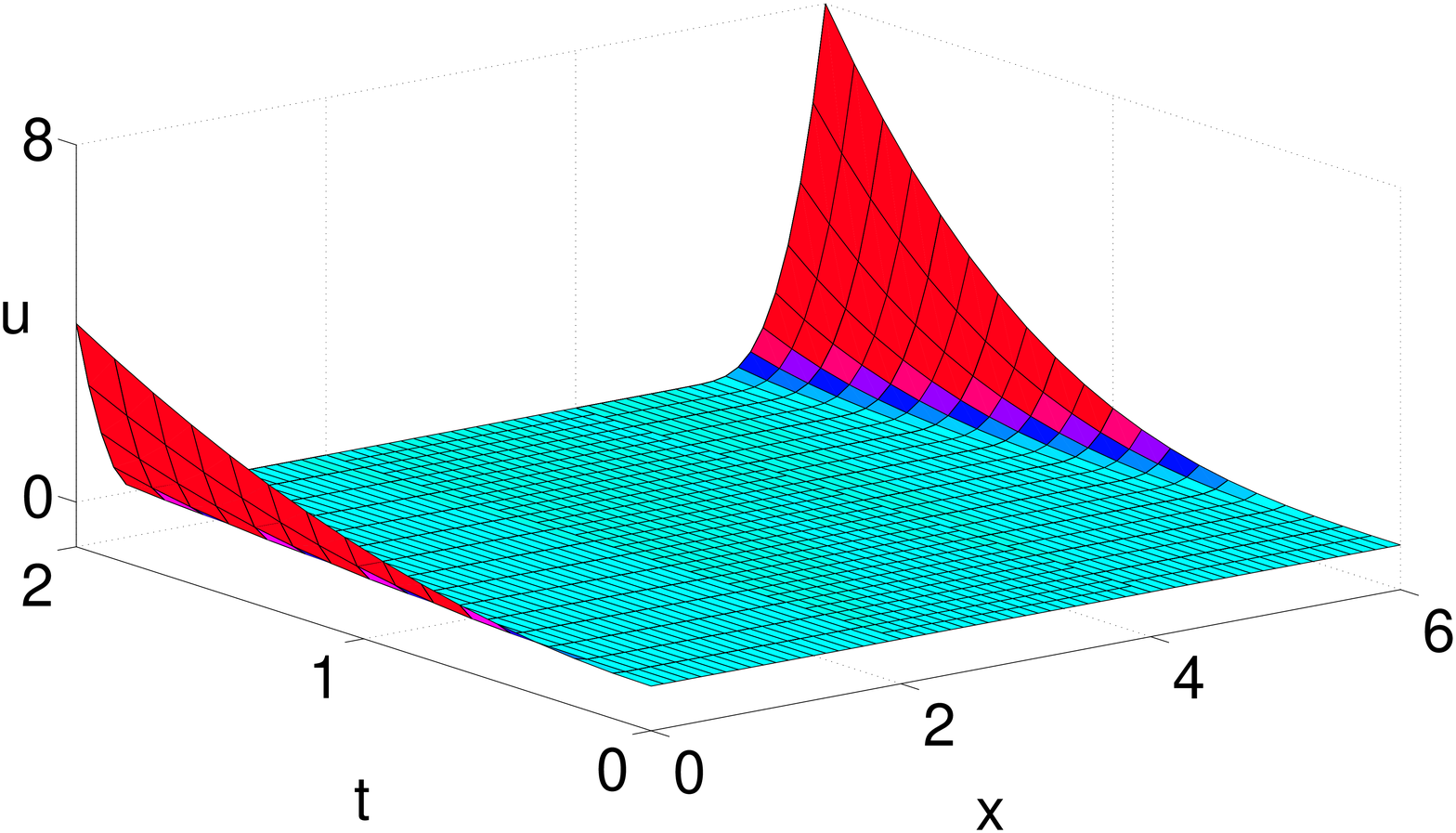}
  \caption{Solution using NNWR method for the wave equation with non-uniform time steps for $\theta=1/4$ for $T=2$: Solution after 1st iteration on the left, and solution after 2nd iteration on the right}
  \label{NumFig11}
\end{figure}

Finally we raise the issue of scalability of the NNWR algorithm by
giving some numerical examples for the wave equation. From Theorem
\ref{NNwaveMulti}, one can say that as long as $h/T$ is constant,
we expect identical convergence behavior of the NNWR algorithm. We
plot in Figure \ref{FigScalable} 
the convergence curves by doubling
the number of subdomains and making the time window length half. One
can therefore conclude that the NNWR algorithm is weakly scalable for the
wave equation.
\begin{figure} 
\centering
\includegraphics[width=0.52\textwidth]{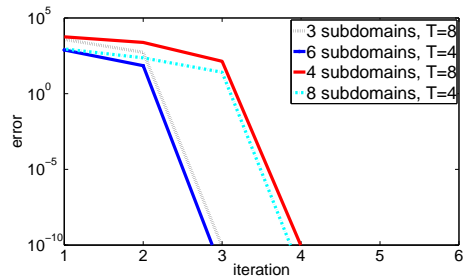}
\caption{Graphs for comparing scalability of NNWR method for the wave equation} \label{FigScalable} 
\end{figure} 

\section{Conclusions }

We defined the NNWR algorithm for multiple subdomains
and for general hyperbolic problems, and analyzed their convergence
properties for the second order wave equation in 1D. We showed
using numerical experiments that for a particular choice of the relaxation
parameter, more specifically for $\theta=1/4$, convergence can be achieved in a finite number of steps.
In fact, this algorithm can be used as a two-step method, choosing
the time window lengh $T$ small enough. We have also extended the
NNWR algorithm for the second order wave equation in 2D, and analyzed
its convergence properties. We have also shown using numerical experiments
that among the DNWR (see \cite{GKM2}) and NNWR methods, NNWR converges faster. But in comparison
to DNWR, the NNWR has to solve twice the number of subproblems (once
for Dirichlet subproblems, and once for Neumann subproblems) on each
subdomain at each iteration. Therefore the computational cost is almost
double for the NNWR than for the DNWR algorithm at each step. However,
we get better convergence behavior with the NNWR in terms of iteration
numbers. Finally we presented a comparison of performences between
the DNWR, NNWR and Schwarz WR methods, and showed that the DNWR and
NNWR converge faster than optimized SWR at least for higher dimensions. 
\afterpage{\clearpage}
\section*{Acknowledgement}
I would like to express my gratitude to Prof. Martin J. Gander and Prof. Felix Kwok for their constant support and stimulating suggestions.

\bibliographystyle{siam}
\bibliography{papersingle}


\end{document}